\documentclass{article}

\usepackage{graphicx}
\usepackage{amssymb}
\usepackage[pdftex]{hyperref}
\usepackage{xcolor}
\hypersetup{
	colorlinks,
	linkcolor={red!50!black},
	citecolor={blue!90!black},
	urlcolor={blue!80!black}
}
\usepackage{cite}
\usepackage[hmargin={2.5cm,2.5cm}, vmargin={3cm,3cm}, dvips]{geometry}

\usepackage{empheq}
\usepackage{fancyhdr}
\usepackage{float}
\definecolor{abbred}{RGB}{255,0,15}
\definecolor{vlg}{RGB}{242,242,242}

\newcommand{\R}{\mathbb{R}}

\newcommand{\E}{\boldsymbol{E}}

\newcommand{\x}{\boldsymbol{x}}
\newcommand{\y}{\boldsymbol{y}}

\renewcommand{\hat}{\widehat}

\usepackage{amsmath} 
\usepackage{amssymb} 
\usepackage{amsthm} 
\usepackage{array}
\usepackage{placeins}
\usepackage{enumerate}
\usepackage{esint}
\usepackage{graphics}
\usepackage{graphicx}
\usepackage{pifont}
\usepackage{tikz}
\usetikzlibrary{3d}
\usetikzlibrary{arrows,calc}
\usetikzlibrary{patterns}
\usetikzlibrary{shapes,positioning,shadows,trees}
\usetikzlibrary{arrows.meta,chains}

\usepackage{pgfplots}
\pgfplotsset{compat=1.15}
\usepackage{caption}
\usepackage{subcaption}
\usepackage{listings}
\usepackage[shortlabels]{enumitem}
\usepackage{multirow}
\usepackage{multicol}
\usepackage{wrapfig}

\usepackage{afterpage}
\usepackage{mathrsfs}
\usepackage{mathtools}
\usepackage{tabularx}

\newtheoremstyle{droit}
{}
{}
{\upshape}
{}
{\bfseries}
{}
{ }
{}
\newtheoremstyle{italique}
{}
{}
{\itshape}
{}
{\bfseries}
{}
{ }
{}
\theoremstyle{italique}
\newtheorem{theorem}{Theorem}[section]
\newtheorem{proposition}[theorem]{Proposition}
\newtheorem{lemma}[theorem]{Lemma}

\theoremstyle{droit}

\newtheorem{remark}[theorem]{Remark}

\newtheorem{assumption}[theorem]{Assumption}

\usepackage{mathtools}
\mathtoolsset{showonlyrefs=true}

\newcommand{\RN}[1]{%
	\textup{\uppercase\expandafter{\romannumeral#1}}%
}

\makeatletter
\newcommand{\vast}{\bBigg@{4}}
\newcommand{\Vast}{\bBigg@{5}}
\makeatother
\usepackage{geometry}
\usetikzlibrary{shapes,backgrounds}

\usepackage{algorithm,algcompatible}
\usepackage{algpseudocode}
\algnewcommand{\lst}{\texttt{lst}}
\algnewcommand{\slst}{\texttt{slst}}
\algnewcommand{\SEND}{\textbf{send}}
\newsavebox{\algleft}
\newsavebox{\algright}

\definecolor{darkgreen}{rgb}{0,0.4,0} 
\definecolor{darkbrown}{rgb}{0.5, 0.396, 0.09}

\definecolor{c1}{rgb}{0.0, 0.4196078431372549, 0.6431372549019608}
\definecolor{c2}{rgb}{1.0, 0.5019607843137255, 0.054901960784313725}
\definecolor{c3}{rgb}{0.6705882352941176, 0.6705882352941176,
	0.6705882352941176} \definecolor{c}{rgb}{0.34901960784313724, 0.34901960784313724, 0.34901960784313724}
\definecolor{c4}{rgb}{0.37254901960784315, 0.6196078431372549,
	0.8196078431372549} 
\definecolor{c5}{rgb}{0.5372549019607843, 0.5372549019607843,
	0.5372549019607843} 
\definecolor{c6}{rgb}{1.0, 0.7372549019607844, 0.4745098039215686}
\definecolor{c7}{rgb}{0.8117647058823529, 0.8117647058823529,
	0.8117647058823529}

\newsavebox{\imagebox}

\usepackage{tikz-3dplot}
\tikzset{declare function={
		vcrity(\ph,\th)=atan2(sin(\th)*sin(\ph),min(cos(\ph),-1/sqrt(2))*cos(\th));
		vcritz(\ph,\th)=\ph;
	},pics/ycylinder/.style={code={
			\tikzset{3d/cylinder/.cd,#1}
			\def\pv##1{\pgfkeysvalueof{/tikz/3d/cylinder/##1}}
			\pgfmathsetmacro{\vmin}{vcrity(\tdplotmainphi,\tdplotmaintheta)}
			\pgfmathsetmacro{\vmax}{\vmin-180}
			\path[3d/cylinder/mantle]
			let \p1=($(0,1,0)-(0,0,0)$),\n1={atan2(\y1,\x1)} in
			[shading angle=\n1]
			plot[variable=\t,domain=\vmin:\vmax,smooth]
			({\pv{r}*cos(\t)},0,{\pv{r}*sin(\t)})
			-- 
			plot[variable=\t,domain=\vmax:\vmin,smooth]
			({\pv{r}*cos(\t)},\pv{h},{\pv{r}*sin(\t)})
			--cycle;
			\pgfmathtruncatemacro{\itest}{sign(cos(\tdplotmainphi))}
			\ifnum\itest=-1
			\path[3d/cylinder/top] plot[variable=\t,domain=0:360,smooth cycle]
			({\pv{r}*cos(\t)},\pv{h},{\pv{r}*sin(\t)}) ;
			\fi
			\ifnum\itest=1
			\path[3d/cylinder/top] plot[variable=\t,domain=0:360,smooth cycle]
			({\pv{r}*cos(\t)},0,{\pv{r}*sin(\t)}) ;
			\fi
	}},3d/.cd,cylinder/.cd,r/.initial=1,h/.initial=1,
	mantle/.style={draw},top/.style={draw}}

\makeatletter
\def\customrevertcolormap#1{%
	\pgfplotsarraycopy{pgfpl@cm@#1}\to{custom@COPY}%
	\c@pgf@counta=0
	\c@pgf@countb=\pgfplotsarraysizeof{custom@COPY}\relax
	\c@pgf@countd=\c@pgf@countb
	\advance\c@pgf@countd by-1 %
	\pgfutil@loop
	\ifnum\c@pgf@counta<\c@pgf@countb
	\pgfplotsarrayselect{\c@pgf@counta}\of{custom@COPY}\to\pgfplots@loc@TMPa
	\pgfplotsarrayletentry\c@pgf@countd\of{pgfpl@cm@#1}=\pgfplots@loc@TMPa
	\advance\c@pgf@counta by1 %
	\advance\c@pgf@countd by-1 %
	\pgfutil@repeat
}%

\makeatother
\usepgfplotslibrary{colorbrewer}
\usepgfplotslibrary{fillbetween}
\usepgfplotslibrary{groupplots}

\makeatletter
\newtoks\pgf@ps@toks
\newcount\c@pgf@ps

\def\pgf@ps@sp{ }

\def\pgf@ps@esettoks#1{\edef\pgf@ps@tmp{#1}\pgf@ps@toks\expandafter{\pgf@ps@tmp}}

\def\pgf@ps@repop#1#2{%
	\c@pgf@countb=#2\relax%
	\def\pgf@ps@op{#1}%
	\def\pgf@ps@ops{}\pgf@ps@@repop}
\def\pgf@ps@@repop{%
	\ifnum\c@pgf@countb<1\relax%
	\else%
	\edef\pgf@ps@ops{\pgf@ps@op\pgf@ps@ops}%
	\advance\c@pgf@countb by-1\relax%
	\expandafter\pgf@ps@@repop%
	\fi%
}

\def\pgf@ps@generate@ps{%
	\c@pgf@counta=\pgf@ps@ncol\relax%
	\c@pgf@countb=\c@pgf@counta%
	\advance\c@pgf@countb by-1\relax%
	\pgf@ps@esettoks{ \noexpand\pgf@ps@interp{col@\the\c@pgf@counta}{col@\the\c@pgf@countb} }%
	\pgfmathloop
	\ifnum\c@pgf@counta<2\relax%
	\else%
	\c@pgf@countb=-\c@pgf@counta%
	\advance\c@pgf@countb by\pgf@ps@ncol\relax%
	\advance\c@pgf@counta by-1\relax%
	\pgf@ps@repop{pop\pgf@ps@sp}{\c@pgf@countb}%
	\c@pgf@countb=\c@pgf@counta%
	\advance\c@pgf@countb by-1\relax%
	\ifnum\c@pgf@countb=0\relax%
	\c@pgf@countb=\pgf@ps@ncol\relax%
	\fi%
	\pgf@ps@esettoks{ \the\c@pgf@counta\pgf@ps@sp eq { \pgf@ps@ops \noexpand\pgf@ps@interp{col@\the\c@pgf@counta}{col@\the\c@pgf@countb} }{ \the\pgf@ps@toks } ifelse}%
	\repeatpgfmathloop%
	\c@pgf@counta=\pgf@ps@ncol\relax%
	\advance\c@pgf@counta by-2\relax%
	\pgf@ps@repop{dup\pgf@ps@sp}{\c@pgf@counta}%
	\pgf@ps@esettoks{ \pgf@ps@ops \the\pgf@ps@toks }%
}

\def\pgf@ps@colorstorgb#1{%
	\c@pgf@ps=1\relax%
	\pgfutil@for\pgf@ps@:={#1}\do{%
		\pgf@ps@coltorgb{\pgf@ps@}{col@\the\c@pgf@ps}%
		\advance\c@pgf@ps by1\relax}%
}

\def\pgf@ps@coltorgb#1#2{%
	\edef\pgf@ps@marshal{\noexpand\pgfshadecolortorgb{#1}}%
	\expandafter\pgf@ps@marshal\expandafter{\csname#2\endcsname}%
}

\def\pgf@ps@rgb#1{\csname#1\endcsname}
\def\pgf@ps@interp#1#2{%
	\pgf@ps@rgb{#1red} mul exch \pgf@ps@rgb{#2red} mul add
	5 1 roll
	\pgf@ps@rgb{#1green} mul exch \pgf@ps@rgb{#2green} mul add
	3 1 roll
	\pgf@ps@rgb{#1blue} mul exch \pgf@ps@rgb{#2blue} mul add
}

\def\pgfdeclarecolorwheelshading#1#2#3{%
	\pgf@ps@getcols{#3}%
	\pgfmathparse{mod(#2+360/\pgf@ps@ncol-90,360)}%
	\pgf@x=\pgfmathresult pt\relax%
	\ifdim\pgf@x<0pt\relax%
	\advance\pgf@x by360pt\relax%
	\fi%
	\edef\pgf@ps@rot{\pgfmath@tonumber{\pgf@x}}%
	\pgf@ps@generate@ps%
	\pgf@ps@esettoks{%
		\noexpand\pgfdeclarefunctionalshading[#3]{#1}%
		{\noexpand\pgfpoint{-50bp}{-50bp}}{\noexpand\pgfpoint{50bp}{50bp}}%
		{\noexpand\pgf@ps@colorstorgb{#3}}%
		{%
			2 copy abs exch abs add 0.0001 ge { atan } { pop } ifelse
			\pgf@ps@rot\pgf@ps@sp add dup 360 ge { -360 add } { } ifelse
			360 \pgf@ps@ncol\pgf@ps@sp 
			div div dup floor dup 3 1 roll neg add dup neg 1 add exch
			2 copy 2 copy 7 -1 roll 1 add
			\the\pgf@ps@toks}}%
	\edef\pgf@ps@marshal{\the\pgf@ps@toks}%
	\pgf@ps@marshal}

\def\pgf@ps@getcols#1{%
	\c@pgf@ps=0\relax%
	\pgfutil@for\pgf@ps@:={#1}\do{\advance\c@pgf@ps by1}%
	\edef\pgf@ps@ncol{\the\c@pgf@ps}%
}
\pgfdeclarecolorwheelshading{rainbow}{-45}{c4,c,c2}

\begin{document}

\title{Towards analysis-aware geometry defeaturing for inception voltage predictions}

\author{Ondine Chanon$^{1}$\\ \\\vspace{-0.2cm}
  \parbox{\linewidth}{\centering\footnotesize{$^1$ ABB Switzerland Ltd, Corporate Research Center,\\
  Segelhofstrasse 1K, 5405 Baden-D\"{a}ttwil, Switzerland.\\
  \href{mailto:ondine.chanon@ch.abb.com}{ondine.chanon@ch.abb.com}.}
  }
}  
\date{June 2026}

\maketitle
\vspace{-0.8cm}
\noindent\rule{\linewidth}{0.4pt}
\thispagestyle{fancy}
\begin{abstract}
Geometry simplification (or \emph{defeaturing}) is routinely employed in numerical simulations of medium- and high-voltage equipment, not only to reduce computational costs but also to make meshing possible at all. However, the errors that this practice introduces in the predicted breakdown or inception voltages are currently neglected. This work presents a goal-oriented \textit{a posteriori} error estimation framework that aims at assessing such defeaturing errors in inception voltage computations.

The methodology combines a first-order approximation of the inception voltage error with dual-weighted residual estimators, providing an error bound that allows us to quantify the impact of defeaturing on the simulation results. The approach builds upon the streamer integral model for inception voltage prediction and uses the recently proposed certified goal-oriented analysis-aware defeaturing estimator of~\cite{weder2025certified} for elliptic PDEs. The first-order approximation of the inception voltage is a linear functional~$J$ of the background electric field whose definition involves a line integral along the critical field line, and is therefore only of low regularity. To meet the abstract regularity assumption of the goal-oriented theory, we introduce a Gaussian mollification of~$J$.
The methodology is illustrated on a pin--plate benchmark with a small protrusion of varying size and shape, using a single shared adaptive mesh and an adaptive coupling of the mollification width to the local mesh size in order to ensure that the defeaturing error dominates over the discretization and the regularization errors.

This study is a first step towards the application of analysis-aware defeaturing to an industrial application, mainly inception voltage predictions. While it operates under quite restrictive assumptions, it also establishes a framework that can be extended to more complex configurations, and which also has potential for application to other breakdown voltage prediction models and to other types of simulations.
\end{abstract}


\section{Introduction}\label{s:intro}
The design of medium- and high-voltage equipment relies on numerical simulations to ensure safe operation. Modern designs contain numerous geometric features at various scales: sharp edges, rounded corners, mounting screws, surface textures, and manufacturing imperfections. When feasible, including all these features in numerical models leads to prohibitively expensive simulations because of meshing requirements and computational cost. Engineers therefore routinely simplify geometries by removing small features deemed negligible, a practice referred to in the literature as \emph{geometry simplification}, \emph{design simplification}, or \emph{defeaturing}; we use these three terms as synonyms throughout this work. Current industrial practice relies largely on user intuition to decide which features to retain or remove. This raises a critical question: how can we assess whether a given defeaturing is acceptable, and what error does it introduce in the predicted inception voltage? To answer this question, we need a systematic, quantitative framework for assessing defeaturing errors that allows engineers to identify which geometric features matter to obtain accurate simulation results.

A key quantity of interest in the design of medium- and high-voltage equipment is the inception voltage, which marks the threshold at which partial discharges occur in gas-insulated regions with non-negligible probability. Many simulation tools support inception voltage prediction by coupling background field computations with discharge modeling based on the streamer integral, see~\cite{vhvlab,hjortsam2017si,thesis:xeno}, but the errors introduced by geometry simplification are currently neither taken into account nor quantifiable. Accurate inception voltage predictions require detailed geometric models to capture the electric field distribution near critical electrode surfaces, and even very small geometric features may be critical.

The mathematical study of geometry simplification errors in elliptic PDEs has recently been developed in the so-called analysis-aware defeaturing literature. Energy-norm \textit{a posteriori} estimators that only depend on local computations on the simplified geometry and along the boundary of the removed feature have been developed in~\cite{paper1defeaturing,buffa2022adaptive,buffa2024equilibrated,paper3multifeature,weder2025analysis}, and have very recently been extended to certified goal-oriented estimators for linear functionals of the solution that satisfy abstract regularity assumptions, see~\cite{weder2025certified}. In parallel, goal-oriented \textit{a posteriori} error estimation through the dual-weighted residual (DWR) technique is by now a well-established methodology for controlling the discretization error in quantities of interest, see~\cite{becker2001optimal,bangerth2003adaptive,giles2002adjoint,prudhomme1999goal}. To the best of the author's knowledge, however, no quantitative framework that combines these ingredients to control geometry simplification errors in inception voltage predictions has been proposed so far.

This work develops such a framework by combining a perturbation analysis of the streamer integral model with the goal-oriented analysis-aware defeaturing estimator of~\cite{weder2025certified}. The inception voltage error caused by removing a small feature is first approximated to first order by a linear functional~$J$ of the background electric field, defined as an integral along the critical field line. This functional turns out to be of low regularity (only in $H^{-3/2-\epsilon}$, while the abstract goal-oriented theory of~\cite{weder2025certified} requires $H^{-1}$-regularity), and we address this difficulty by introducing a Gaussian mollification of~$J$. The methodology yields a goal-oriented defeaturing estimator that requires only computations on the simplified geometry and along the boundary of the removed feature, avoiding the need to mesh or solve the full detailed geometry. We further show, both theoretically and through numerical experiments, that the estimator captures the correct scaling of the inception voltage error with respect to the feature size. We also assess numerically its effectivity across features of varying shapes.

This study operates under quite restrictive assumptions and focuses on a simple but industrially relevant pin--plate two-electrode configuration to validate the methodology. To the best of our knowledge, this work is a first step towards the application of analysis-aware defeaturing to an industrial application, mainly inception voltage predictions. It establishes a framework that can be extended to more complex industrial geometries, and which could be applied to other breakdown voltage prediction models as well as to other types of simulations.

The manuscript is structured as follows. Section~\ref{s:models} reviews the inception voltage calculation based on the streamer integral model and introduces the working assumptions and notation. Section~\ref{s:framework} presents the error estimation framework, deriving both the first-order approximation of the inception voltage error and the associated goal-oriented defeaturing estimator, including the treatment of the regularity issue by Gaussian mollification. Section~\ref{s:numexample} validates the methodology through numerical experiments on a two-dimensional pin--plate configuration with a small protrusion, assessing the dependence of the estimator on both the feature size and shape. Section~\ref{s:conclusion} summarizes the findings and outlines directions for future work. Supporting proofs are collected in the Appendix, and the notation used throughout the manuscript is summarized in Table~\ref{tbl:notation}.

\begin{table}[H]
\begin{tabularx}{\textwidth}{l|X}
	Notation & Definition \\ \hline
    $\partial M$ & Boundary of $M$. If $M$ is unbounded, $\partial M$ does not include infinity.\\
	$\overline{M}$ & Closure of $M$.\\
	$\mathrm{int}(M)$ & Interior of $M$.\\
    $z\vert_{M}$ & Restriction of function $z$ to $M$.\\
	$\|z\|_{0,M}$ & Norm of $z$ in $L^2(M)$.\\
    $H^s(M)$ & Sobolev space of order $s$ on $M$, see~\cite{adams2003sobolev,grisvard}.\\
    $H^1_{0}(M)$ & Set of functions $z\in H^1(M)$ with zero trace on $\partial M$, i.e., such that $z=0$ on $\partial M$.\\
    $\mathcal F^*$ & Dual space of $\mathcal F$.\\
    $|\boldsymbol{v}|$ & Magnitude (or equivalently, $\ell^2$-norm or Euclidean norm) of $\boldsymbol{v}$.
\end{tabularx}
\caption{Notation used throughout the manuscript, where $s\in\R$, $M$ is an open $k$-dimensional manifold in $\R^3$, $k\le 3$, $\mathbf{v}$ is a vector, $z$ is a function, and $\mathcal F$ is a normed functional space. }\label{tbl:notation}
\end{table}

\section{Inception voltage calculation}\label{s:models}
In this section, we introduce the mathematical framework allowing us to compute the partial discharge inception voltage $\hat V$ for a given switchgear design. In Section~\ref{ss:streamerintegralmodel}, we introduce the general framework and we describe a commonly used criterion to determine if inception may occur from a given point $\x^\star$ on the boundary of the gas volume in the switchgear. We mostly use this section to introduce the main physical concepts and mathematical notations, and to make explicit all (functional) dependencies which will be needed in the rest of this manuscript. For more details on the physical modelling aspects, please refer for instance to~\cite{vhvlab, hjortsam2017si} and references therein. In Section~\ref{ss:modelwithreducingassumptions}, we make some assumptions to reduce the problem to a simpler two-electrode geometry without insulators nor floating conductors. This case will be used in the rest of this document to analyse the geometry simplification error introduced when one computes the inception voltage $\hat V_0$ for a simplified switchgear design, rather than for the exact original one.

\subsection{Streamer integral model of inception}\label{ss:streamerintegralmodel}
Let
\begin{itemize}
    \item $D^G$ be the grounded electrode\footnotemark[1],
    \item $D^{V}$ be the high voltage electrode\footnotemark[1] which does not touch the grounded electrode\footnotemark[2],
    \item $D^F_j$ be $J$ floating conductors\footnotemark[1], with $j=1,\ldots,J$, and let $D^F$ be their union\footnotemark[3],
    \item $D^I_i$ be $I$ insulated volumes\footnotemark[1] with constant (positive) electric permittivity $\varepsilon_i$, with $i=1,\ldots,I$, and let $D^I$ be their union\footnotemark[3],
    \item $D^g$ be the gas volume filling the rest of the space\footnotemark[4], whose gas medium has constant (positive) electric permittivity~$\varepsilon_g$.
\end{itemize}
Figure~\ref{fig:illustrationdomains} gives an illustration of this notation.
\footnotetext[1]{open, bounded, and smooth subdomains of $\R^3$.}
\footnotetext[2]{i.e., $\overline{D^V}\cap\overline{D^G}=\emptyset$.}
\footnotetext[3]{i.e., $D^F := \mathrm{int}\left(\bigcup_{j=1}^J \overline{D_F^j}\right)$ and $D^I := \mathrm{int}\left(\bigcup_{i=1}^I \overline{D_I^i}\right)$.}
\footnotetext[4]{i.e., $D^g := \R^3 \setminus \left( \overline{D^G} \cup \overline{D^{V}} \cup \overline{D^I} \cup \overline{D^F} \right)$, and thus it is an open, unbounded, and not necessarily Lipschitz domain.}

\begin{figure}
\begin{center}
\begin{tikzpicture}[scale=1.2]
\definecolor{electrodecolor}{RGB}{200,100,100}
\fill[pattern=north east lines, pattern color=black!50] (0.2,0.2) rectangle (9.8,5.8);
\node at (8.5,5) {\color{black!60}$D^g$};
\fill[gray!40] (0.6,3.5) rectangle (5,5);
\node at (2.8,4.25) {$D^I_1$};
\fill[gray!40] (6,1.5) circle (0.6);
\node at (6,1.5) {$D^I_2$};
\begin{scope}
  \path[fill=electrodecolor!60, draw=none]
    (1.6,3.0)                            
    ++(60:0.95)                          
    arc[start angle=60, end angle=-60, radius=0.95]   
    -- ++(120:0.25)                      
    arc[start angle=-60, end angle=60, radius=0.70]   
    -- cycle;
  \node at (2.65,2.4) {\color{electrodecolor}$D^F_1$};
\end{scope}
\begin{scope}
  \path[fill=electrodecolor!60, draw=none]
    (7.8,3.0)                            
    ++(120:0.95)                         
    arc[start angle=120, end angle=240, radius=0.95]  
    -- ++(60:0.25)                       
    arc[start angle=240, end angle=120, radius=0.70]  
    -- cycle;
  \node at (6.7,2.4) {\color{electrodecolor}$D^F_2$};
\end{scope}
\fill[electrodecolor!60] (0.6,1.5) rectangle (1.6,4.5);
\node at (1.1,3) {$D^G$};
\fill[electrodecolor!60] (8.4,3.6) arc[start angle=90,end angle=270,radius=0.6] -- cycle;
\fill[electrodecolor!60] (8.4, 2.4) rectangle (9.5, 3.6);
\node at (8.7,3) {$D^{V}$};
\end{tikzpicture}
\end{center}
\caption{2D-illustration of the domains: gas volume $D^g$ (hashed), insulators $D_1^I$ and $D^1_2$ (in gray), grounded electrode $D^G$, high voltage electrode $D^{V}$, and floating electrodes $D^F_1$ and $D^F_2$ (conductors are in red).}\label{fig:illustrationdomains}
\end{figure}
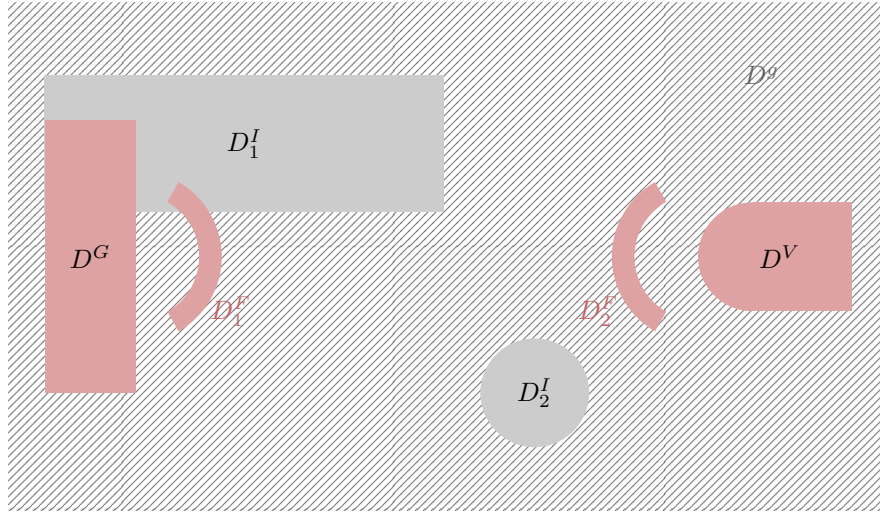

From Faraday's law of induction, the \textit{(background) electric field} $\overline{\E}:\R^3\to\R$ satisfies $\nabla \times \overline{\E} = \mathbf 0$ under the magneto-static assumption. Therefore, there exists a \textit{background electric scalar potential} $\overline{\varphi}:\R^3\to\R$ such that
\begin{equation}\label{eq:defbackgroundelectricfield}
    \overline{\E} = -\nabla \overline{\varphi}.
\end{equation}
Moreover, assuming that the media are net-neutral (i.e., before any electric breakdown happens) with electric permittivity $\varepsilon$, then from Gauss' law for electricity, $\nabla \cdot (\varepsilon \overline{\E}) = 0$. Consequently, combining this with equation~\eqref{eq:defbackgroundelectricfield},
\begin{equation}
    -\nabla\cdot(\varepsilon \nabla \overline{\varphi}) = 0.
\end{equation}
We furthermore assume that both electrodes and all the floating conductors are perfect conductors and/or are not carrying any current, meaning that $\overline{\varphi}$ is constant on $D^G$, $D^{V}$ and $D^F_j$ for each $j=1,\ldots,J$. Since $\overline{\varphi}$ is only well-defined up to a constant from~\eqref{eq:defbackgroundelectricfield}, we assume without loss of generality that it is equal to zero at the grounded electrode $D^G$. In addition, we normalize the problem by the electric potential at $D^{V}$ (or equivalently by the voltage between the two electrodes). 

To summarize, if we define $\varepsilon:D^g\cup D^I\to\R$ as
\begin{equation}
    \varepsilon(x) = \begin{cases}
        \varepsilon_g & \text{for all }x\in D^g,\\
        \varepsilon_i & \text{for all }x\in D^I_i \text{ and all }i=1,\ldots,I,\\
    \end{cases}
\end{equation}
then the background electric potential $\overline{\varphi}$ satisfies the problem
{\mathtoolsset{showonlyrefs=false}
\begin{subequations}
    \begin{empheq}[left=\empheqlbrace]{align}
        &-\nabla\cdot(\varepsilon \nabla \overline{\varphi}) = 0 && \hspace{-4cm}\text{ in } D^g\cup D^I, \label{eq:poissonmain}\\
        &\nabla \overline{\varphi} = 0 && \hspace{-4cm}\text{ in }  D^F,\label{eq:poissonDF}\\
        &\overline{\varphi} = 0 && \hspace{-4cm}\text{ in } D^G, \label{eq:poissonDG}\\
        &\overline{\varphi} = 1 && \hspace{-4cm}\text{ in } D^{V},\label{eq:poissonDV}\\
        &\left[\overline{\varphi}\right] = \left[\varepsilon \nabla\overline{\varphi}\cdot \mathbf{n}\right] = 0 && \hspace{-4cm}\text{ on } \partial D^g\cap\partial D^I,\label{eq:poissoncontinuity}\\
        &\displaystyle\int_{\partial D_j^F} \left[\varepsilon \nabla \overline{\varphi}\cdot \mathbf{n}_j\right] \,\mathrm ds = 0 && \hspace{-4cm}\text{ for each } j=1,\ldots,J,\label{eq:poissonclosure}
    \end{empheq}
\end{subequations}}
where \begin{itemize}
    \item $[\cdot]$ denotes the jump operator for a discontinuous function across an interface, defined as the difference between the function values on both sides,
    \item $\mathbf{n}$ denotes the unit normal vector pointing outward from $D^g$ on its boundary, 
    \item $\mathbf{n}_j$ denotes the unit normal vector pointing outward from $D_j^F$ on its boundary. 
\end{itemize}
Equation~\eqref{eq:poissoncontinuity} ensures the continuity of the background electric potential and corresponding field. Moreover, equation~\eqref{eq:poissonclosure} is a required (scalar) closure equation, stating that the floating conductors have no net charges, i.e., they are initially uncharged and have no connection to external circuits. Note that the domain in which Gauss' law is solved is made of the gas volume $D^g$ and the insulated volumes~$D^I$.

Since $\varepsilon_g$ and $\varepsilon_i$ are constant for all $i=1,\ldots,I$, note that $\overline{\varphi}$ is harmonic in $D^g$ and in $D^I_i$ for all $i=1,\ldots,I$ (but not in their union). Consequently, by the maximum principle, both the minimum and the maximum of $\overline{\varphi}$ are attained at the boundary of one of these domains. Moreover, from Lemma~\ref{lemma:subharmonicgradient}, $\left|\nabla \overline{\varphi}\right|=\left|\overline{\E}\right|$ is sub-harmonic in each domain. Therefore, by the maximum principle again, the maximum of the electric field magnitude $\left|\overline{\E}\right|$ is also attained on one of these boundaries.

Remark that both problem~\eqref{eq:poissonmain}--\eqref{eq:poissonclosure} and the functional dependency of $\overline{\E}$ on $\overline{\varphi}$ from~\eqref{eq:defbackgroundelectricfield} are linear. Therefore, let us call \textit{electric field} (resp. \textit{electric potential}) any field $\E:\R^3\to\R$ (resp. $\varphi:\R^3\to\R$) obtained from the background electric field (resp. background electric potential) re-scaled by a positive scalar $\omega$, i.e.,
\begin{equation} \label{eq:electricfieldpotentialrelation}
    \E := \omega\overline{\E} = -\omega\nabla\overline{\varphi} =: -\nabla \varphi.
\end{equation}
Note that for all $\x\in D^{V}$, since $\overline{\varphi}(\x)=1$, then $\varphi(\x)=\omega$ for all scalars $\omega$. In other words, when problem~\eqref{eq:poissonmain}--\eqref{eq:poissonclosure} is scaled by $\omega$, then the voltage between the two electrodes is equal to $\omega$.\\

In the following, our objective is to determine the \textit{inception voltage}, i.e., the minimum voltage required between the two electrodes for inception to occur in $D^g$. To precisely define it, we first need to introduce the notion of \textit{electric field line} restricted to the gas volume $D^g$. An \textit{electric field line} originating at a given point $\x^\star\in\partial D^g$ and restricted to $\overline{D^g}$, denoted by $\nu\subset\overline{D^g}$, is the curve $\x(t)$ transported by the electric field $\pm\overline{\E}$, where $t$ parameterizes the curve in both directions from $\x^\star$. At insulator boundaries, the field line follows the tangential electric field component. 

More rigorously, let $\overline{\E}^\mathbf{t} := \overline{\E}-(\overline{\E}\cdot\mathbf{n})\mathbf{n}$ be the tangential part of $\overline{\E}$ on $\partial D^g$. Then the electric field line $\nu$ is defined by
\begin{align}\label{eq:deffieldlines}
    \nu := \Bigg\{ \x(t)\in\overline{D^g}: &\,t\in\R, \x(0) = \x^\star, \\
        &\displaystyle\frac{\mathrm d \x}{\mathrm d t}(t) = \displaystyle\frac{\overline{\E}\big(\x(t)\big)}{\big|\overline{\E}\big(\x(t)\big)\big|} && \text{if } \x(t)\not\in\bigcup_{i=1}^I\partial D^I_i, \text{ or if } \x(t)\in\bigcup_{i=1}^I\partial D^I_i \text{ and } \mathrm{sign}(t)\overline{\E}\big(\x(t)\big)\cdot\mathbf{n}< 0,\\
        &\displaystyle\frac{\mathrm d \x}{\mathrm d t}(t) = \displaystyle\frac{\overline{\E}^{\mathbf{t}}\big(\x(t)\big)}{\big|\overline{\E}^{\mathbf{t}}\big(\x(t)\big)\big|} && \text{otherwise, i.e., if } \x(t)\in\bigcup_{i=1}^I\partial D^I_i \text{ and } \mathrm{sign}(t)\overline{\E}\big(\x(t)\big)\cdot\mathbf{n}\geq 0 \qquad\;\Bigg\}.
\end{align}
Note that 
\begin{itemize}
    \item to highlight the dependence of the electric field line on the inception point $\x^\star$ and the background electric field $\overline{\E}$, we write $\nu = \nu(\x^\star,\overline{\E})$,
    \item since $\overline{\E}=\boldsymbol{0}$ in the conductors (i.e., in the electrodes and in the floating conductors), the field lines stop as soon as they hit the conductors\footnotemark[5];
    \footnotetext[5]{i.e., whenever $\x(t)\in\partial D^G\cup\partial D^V\cup \bigcup_{j=1}^J \partial D^F_j$ and sign$(t)\overline{\E}\big(\x(t)\big)\cdot \mathbf{n}>0$ for some $t\neq0$.}
    \item since we consider both positive and negative values of $t$, the field lines propagate in $D^g$ in two directions: in the one of (the tangential component of) $\overline{\E}$, and in the one of (the tangential component of) $-\overline{\E}$. Due to the previous remark, a field line initiated at some $\x^\star$ belonging to the boundary of a conductor will only propagate in one of the two directions, depending on the sign of $\overline{\E}(\x^\star)\cdot\mathbf{n}$. Instead, a field line initiated at some $\x^\star$ belonging to the boundary of an insulator will usually propagate in both directions;
    \item since the transport direction $\overline{\E}$ is normalized, the length of the field line from~$\x^\star$ to~$\x(t)$ is equal to $|t|$;
    \item if we define $\E^{\mathbf{t}}$ similarly as $\overline{\E}^{\mathbf{t}}$ but on the scaled electric field, then
        \begin{equation} \label{eq:fieldlinesindepfromscaling}
            \displaystyle\frac{\overline{\E}}{\big|\overline{\E}\big|}=\frac{\E}{\big|\E\big|} \quad \text{ and } \quad \displaystyle\frac{\overline{\E}^{\mathbf{t}}}{\big|\overline{\E}^{\mathbf{t}}\big|}=\frac{\E^{\mathbf{t}}}{\left|\E^{\mathbf{t}}\right|},
        \end{equation}
        and thus the electric field lines are independent from any scaling of the electric potential. 
\end{itemize}

Now, let $\alpha_\text{eff}:\R^+\to\R$ be the effective ionization coefficient function known \textit{a priori}, which takes as argument a local electric field strength (i.e., the norm of the electric field at a given point). This quantity corresponds to the net number of free electrons generated per unit path length, representing the difference between the rate of impact ionization and the rate of electron attachment.
Moreover, let $E_c$ be the \textit{a priori} known magnitude of the \textit{critical electric field}, i.e., the minimal field magnitude from which an electron avalanche can be created. Then we assume that 
\footnotetext[6]{This assumption can be straightforwardly relaxed to the case in which $\alpha_\text{eff}(\xi)$ is continuously differentiable for almost every $\xi>E_c$, i.e., except on a subset of measure zero.}
\begin{itemize}
    \item $\alpha_\text{eff}$ is a continuous function, 
    \item for $\xi>E_c$, $\alpha_\text{eff}(\xi)$ is continuously differentiable\footnotemark[6],
    \item for $\xi>E_c$, $\alpha_\text{eff}(\xi)$ is monotonously increasing, i.e., $\alpha_\text{eff}'(\xi)>0$; in all analyzed practical applications, this assumption is indeed satisfied experimentally,
    \item $\alpha_\text{eff}(\xi)=0$ for $0 \leq \xi \leq E_c$; this assumption ensures that the parts of $\nu$ where the electric field magnitude is below $E_c$ do not contribute to the streamer integral $S$ defined below.
\end{itemize}

Then, let $T^+$ and $T^-$ be positive real numbers or infinity, and define the \textit{streamer integral} $S$ to be the scalar
\begin{equation}
    S = S(\omega, \overline{\E}, \nu, T^-, T^+) := \int_{-T^-}^{T^+} \alpha_\text{eff}\big(\big|\omega\overline{\E}\big(\x(t)\big)\big|\big) \,\mathrm dt,
\end{equation}
where $\x(t)$ are the points along the field line $\nu$. Physically, $S$ corresponds to the size of the electron avalanche created by the electric field along the field line $\nu$. 

A common criterion to determine if inception may occur is the following. Let $K_c$ be the experimentally determined (positive) \textit{streamer constant}. Then inception from point $\x^\star$ is assumed to occur when a voltage $\omega$ is applied between the two electrodes if there exist $T^-$ and $T^+$ such that 
\begin{subequations}
\begin{align}
    &\left|\omega\overline{\E}\big(\x(t)\big)\right|\geq E_c \quad \text{ for all } t\in(-T^-, T^+), \label{eq:inceptioncriteriaa}\\
    \text{and}\quad  &S\big(\omega, \overline{\E}, \nu(\x^\star, \overline{\E}), T^-, T^+\big) \geq K_c.\label{eq:inceptioncriteriab}
\end{align}
\end{subequations}
Physically, this means that inception occurs if the electric field magnitude exceeds the critical electric field over some region of the gas domain, and if the electron avalanche created by the electric field reaches the critical size characterized by $K_c$ needed to become self-sustaining\footnotemark[7].
\footnotetext[7]{More precisely, we require that $\left|\omega\overline{\E}\big(\x(t)\big)\right|>E_c$ for almost every $t\in(-T^-, T^+)$, i.e., except on a subset of measure zero, together with condition~\eqref{eq:inceptioncriteriab}. However, since $\overline{\E}$ is the gradient of a harmonic function and $\alpha_\text{eff}$ is continuous, then our requirements are equivalent to conditions~\eqref{eq:inceptioncriteriaa}--\eqref{eq:inceptioncriteriab}.}

Recall that in the gas volume $D^g$, the maximum of $|\overline{\E}|$ (and thus also of $|\omega\overline{\E}|$) is attained at the boundary $\partial D^g$. Thus for inception to occur, it is enough to have one point $\x^\star\in \partial D^g$ such that $|\omega\overline{\E}(\x^\star)|\geq E_c$, and such that~\eqref{eq:inceptioncriteriaa} and~\eqref{eq:inceptioncriteriab} are satisfied. The inception voltage $\hat V$ is then defined to be the smallest voltage $\omega$ satisfying this criterion.

\subsection{Inception voltage computation under simplifying assumptions} \label{ss:modelwithreducingassumptions}
As highlighted above, the physical problem tackled in this work is complex. In order to make a first step towards a complete analysis of the geometry simplification error, we therefore make a simplifying assumption:
\begin{assumption}\label{as:2elecgeom}
    We consider a two-electrode geometry without insulators nor floating conductors, i.e., $D^I=\emptyset$ and $D^F=\emptyset$. 
\end{assumption}
Under this assumption, problem~\eqref{eq:poissonmain}--\eqref{eq:poissonclosure} reduces to the following simple constant-coefficient Laplace problem in~$D^g$ with piecewise constant Dirichlet boundary conditions:
\begin{equation} \label{eq:poissonproblemwithassumptions}
    \begin{cases}
        -\nabla\cdot(\varepsilon_g \nabla \overline{\varphi}) = 0 & \text{ in } D^g, \text{ and thus } -\Delta \overline{\varphi}=0 \text{ in } D^g,\\
        \overline{\varphi} = 0 & \text{ in } D^G,\\
        \overline{\varphi} = 1 & \text{ in } D^{V}.
    \end{cases}
\end{equation}
Moreover, we have $D^g = \R^3 \setminus \left( \overline{D^G} \cup \overline{D^{V}} \right)$, and since $D^G$ and $D^V$ are smooth and non-intersecting, then $D^g$ is also smooth. Consequently, $\overline{\varphi}$ is also smooth in $D^G$. The analysis in this work requires $\overline{\varphi} \in H^{2+\epsilon}(D^g)$ (see Section~\ref{ss:1storderapprox}). By elliptic regularity, the minimal regularity assumption on the domain required for this work is thus $D^g$ to be $C^3$-regular.

Furthermore, under Assumption~\ref{as:2elecgeom}, the gas volume boundary $\partial D^g$ corresponds to the boundaries of the electrodes $\partial D^V\cup\partial D^G$. Therefore, the field lines $\nu$ originating at some point $\x^\star\in\partial D^g$ only propagate in one direction; i.e., either in the direction of $\overline{\E}$ and thus we can take $T^-=0$, or in the direction of $-\overline{\E}$ and thus we can take $T^+=0$. Consequently, the streamer integral can be expressed as
\begin{equation}
    S(\omega, \overline{\E}, \nu, T) = \int_{0}^{T} \alpha_\text{eff}\big(\big|\omega\overline{\E}\big(\x(\,\tilde t\,)\big)\big|\big) \,\mathrm dt, \;\text{ where } \begin{cases}
        T = T^+, \, \tilde t = t &\text{ if } T^-=0,\\
        T = T^-, \tilde t = -t &\text{ if } T^+=0.
    \end{cases}
\end{equation}
Without loss of generality, let us assume in the following that propagation is always done in the direction of $\overline{\E}$ and thus we take $T^-=0$ and consider $T=T^+$. Consequently,
\begin{equation}\label{eq:defS}
    S(\omega, \overline{\E}, \nu, T) = \int_{0}^{T} \alpha_\text{eff}\big(\big|\omega\overline{\E}\big(\x(t)\big)\big|\big) \,\mathrm dt.
\end{equation}
The steps needed to compute the streamer integral $S$ are summarized in the diagram of Figure~\ref{fig:forwardpbS}.\\

\begin{figure}
    \begin{center}
        \begin{tikzpicture}[
        start chain = going right,
        node distance=7mm,
        block/.style={shape=rectangle, draw,
            inner sep=1mm, align=center,
            minimum height=5mm, minimum width=25mm}] 
        \node[block] (n1) {(Insulators, floating\\conductors and)\\Gas domains};
        \node[block, below=2mm of n1] (n1bis) {Boundary\\conditions};
        \node[block, right=of n1] (n1half) at ($(n1.east)!0.5!(n1.east) - (0,5mm)$) {Background\\electric\\potential $\overline{\varphi}$};
        \node[block, right=of n1half] (n2r1) {Background\\electric field $\overline{\E}$};
        \node[block, above=3mm of n2r1] (n2r2) {Point $\x^\star$};
        \node[block, right=of n2r1] (n3r1) at ($(n2r1.east)!0.5!(n2r2.east) + (4mm,0)$) {Field line $\nu$};
        \node[block, below=6mm of n3r1] (n3r2) {$\omega$};
        \node[block, below=2mm of n3r2] (n3r3) {$T$};
        \node[block, right=10mm of n3r2] (n4) at ($(n3r2.east)!0.5!(n3r3.east) + (0mm,5mm)$) {Streamer \\integral $S$};
        
        \draw[->] (n1.east) -- ($(n1half.west) + (0,1mm)$);
        \draw[->] (n1bis.east) -- ($(n1half.west) + (0,-1mm)$);
        \draw[->] (n1half.east) -- (n2r1.west);
        \draw[->] ($(n2r1.east) + (0,1mm)$) -- ($(n3r1.west) + (0,-1mm)$);
        \draw[->] ($(n2r1.east) + (0,0mm)$) -- ($(n4.west) + (0,1mm)$);
        \draw[->] (n2r2.east) -- ($(n3r1.west) + (0,1mm)$);
        \draw[->] (n3r1.east) -- ($(n4.west) + (0,3mm)$);
        \draw[->] (n3r2.east) -- ($(n4.west) + (0,-1mm)$);
        \draw[->] (n3r3.east) -- ($(n4.west) + (0,-3mm)$);
        \end{tikzpicture}
        \caption{Full forward problem: computation of the streamer integral $S$.}\label{fig:forwardpbS}
    \end{center}
\end{figure}
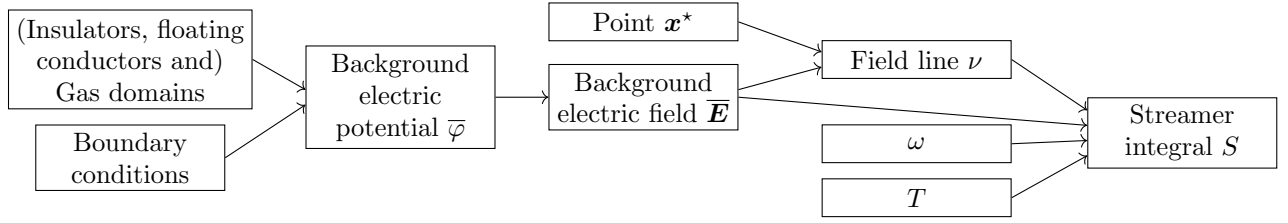

Now, for all points $\x^\star\in \partial D^g$ and for all voltages $\omega$, let $T^\mathrm{\partial}$ be the length of the full field line, i.e., the parameter $t$ at which the field line hits the opposite electrode:
\begin{equation}\label{eq:Tbd}
    T^\mathrm{\partial} := \min\left\{t>0: \x(t)\in\partial D^g\right\}.
\end{equation}
Then, let $T^c$ be the length of the portion of the field line along which the electric field magnitude is supercritical, i.e.,
\begin{equation}\label{eq:Tmax}
    T^c := \begin{cases}
        0 & \text{ if } \big|\omega\overline{\E}(\x^\star)\big|< E_c;\\
        \min\left\{t\in(0,T^\partial): \big|\omega\overline{\E}\big(\x(t)\big)\big| = E_c, \text{ and }\right. & \\
        \hspace{2.75cm} \left. t \text{ is not a local minimum of }\big|\omega\overline{\E}\big(\x(t)\big)\big|\right\} & \text{ otherwise, and if it exists};\\
        T^\mathrm{\partial} &\text{ otherwise.}
    \end{cases}
\end{equation} 
In the first case, the magnitude of the electric field at the field line's orign $\x^\star$ is not supercritical. In the last case, the electric field magnitude is supercritical everywhere along the field line. The second case corresponds to the intermediate situation.
That is, the first criterion for inception~\eqref{eq:inceptioncriteriaa} can only be verified on a subset of the interval $(0,T^c)$. 

Moreover, we call $\hat T = \hat T\big(\omega, \overline{\E}, \nu(\x^\star, \overline{\E})\big)$ the length of the critical portion of the field line, i.e., of the portion of the field line at which the electron avalanche reaches the critical size $K_c$, if it ever reaches it. That is, since $\alpha_\text{eff}$ is positive in the interval $(0,T^c)$, then 
\begin{itemize}
    \item either $S\big(\omega, \overline{\E}, \nu, T^c\big) < K_c$, i.e., the second criterion for inception~\eqref{eq:inceptioncriteriab} is never verified and there cannot be inception; in this case, we define $\hat T := 0$ and thus $S\Big(\omega, \overline{\E}, \nu, \hat T\Big)=0$;
    \item or there exists $\hat T \in (0, T^c]$ such that 
        \begin{equation}\label{eq:SIeqKc}
            S\Big(\omega, \overline{\E}, \nu, \hat T\Big) = K_c.
        \end{equation}
\end{itemize}
Note that we always have $\hat T \leq T^c \leq T^\partial$.

Figure~\ref{fig:alphaplots} illustrates different scenarios occurring in relevant applications, helping to build intuition about $T^c$ and~$\hat T$. The first scenario represents a homogeneous electric field, as in a two parallel-plate electrode configuration. In this case, if the electric field is supercritical in $\x^\star$, then it is supercritical everywhere and thus $T^c=T^\partial$. The second scenario represents a rapidly decaying electric field, as in a pin-plate electrode configuration. And the third scenario represents a more complex electrode configuration, for example, two facing electrodes each creating strong localized electric fields.\\

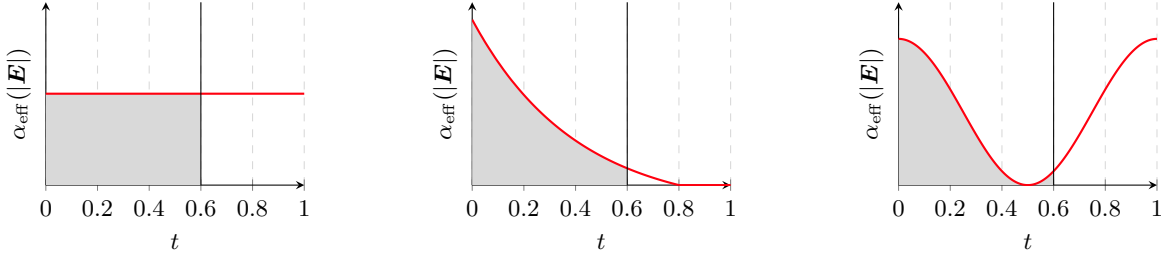
\begin{figure}
    \centering
    \begin{subfigure}[t]{0.32\linewidth}
        \centering
        \begin{tikzpicture}
            \begin{axis}[
                name=plot1,
                width=5cm,
                height=4cm,
                xlabel={$t$},
                ylabel={$\alpha_{\mathrm{eff}}(|\E|)$},
                xmin=0, xmax=1,
                ymin=0, ymax=1,
                grid=major,
                grid style={dashed, gray!30},
                xtick={0,0.2,0.4,0.6,0.8,1},
                ytick=\empty,
                axis lines=left,
                tick label style={font=\small},
                label style={font=\small},
            ]
            \addplot[draw=none, fill=gray!30, domain=0:0.6] {0.5} \closedcycle;
            \addplot[abbred, thick,domain=0:1] {0.5};
            \addplot[black] coordinates {(0.6,0) (0.6,1)};
            \end{axis}
        \end{tikzpicture}
        \caption{Homogeneous electric field for which $T^c=1=T^\partial$.}\label{sfig:homogeneous}
    \end{subfigure}
    \hfill
    \begin{subfigure}[t]{0.32\linewidth}
        \centering
        \begin{tikzpicture}
            \begin{axis}[
                name=plot2,
                at={($(plot1.east)+(1cm,0)$)},
                anchor=west,
                width=5cm,
                height=4cm,
                xlabel={$t$},
                ylabel={$\alpha_{\mathrm{eff}}(|\E|)$},
                xmin=0, xmax=1,
                ymin=0, ymax=1,
                grid=major,
                grid style={dashed, gray!30},
                xtick={0,0.2,0.4,0.6,0.8,1},
                ytick=\empty,
                axis lines=left,
                tick label style={font=\small},
                label style={font=\small},
            ]
            \addplot[draw=none, fill=gray!30, smooth, samples=100, domain=0:0.6] 
                {7*exp(-5*(x/2+0.38))-7*exp(-5*(0.4+0.38)} \closedcycle;
            \addplot[abbred, thick, smooth, samples=100, domain=0:0.8] 
                {7*exp(-5*(x/2+0.38))-7*exp(-5*(0.4+0.38)};
            \addplot[abbred, thick, domain=0.8:1] {0};
            \addplot[black] coordinates {(0.6,0) (0.6,1)};
            \end{axis}
        \end{tikzpicture}
        \caption{Rapidly decaying electric field for which $T^c=0.8$, $T^\partial = 1$.}
    \end{subfigure}
    \hfill
    \begin{subfigure}[t]{0.32\linewidth}
        \centering
        \begin{tikzpicture}
            \begin{axis}[
                name=plot3,
                at={($(plot2.east)+(1cm,0)$)},
                anchor=west,
                width=5cm,
                height=4cm,
                xlabel={$t$},
                ylabel={$\alpha_{\mathrm{eff}}(|\E|)$},
                xmin=0, xmax=1,
                ymin=0, ymax=1,
                grid=major,
                grid style={dashed, gray!30},
                xtick={0,0.2,0.4,0.6,0.8,1},
                ytick=\empty,
                axis lines=left,
                tick label style={font=\small},
                label style={font=\small},
            ]
            \addplot[draw=none, fill=gray!30, smooth, samples=200, domain=0:0.6] 
                {0.4*(cos(deg(2*pi*x)) + 1)} \closedcycle;
            \addplot[abbred, thick, smooth, samples=200, domain=0:1] 
                {0.4*(cos(deg(2*pi*x)) + 1)};
            \addplot[black] coordinates {(0.6,0) (0.6,1)};
            \end{axis}
        \end{tikzpicture}
        \caption{Complex electric field for which $T^c=1=T^\partial$.}
    \end{subfigure}
    \caption{Effective ionization coefficient along a given field line in three different scenarios for which inception may occur. The size of the gray areas correspond to the value of the streamer integral when it is equal to~$K_c$, leading to $\hat T=0.6$ in all three cases. In~(a) and in~(b), inception may still occur under a lower voltage, but not in~(c).}
    \label{fig:alphaplots}
\end{figure}

For any $\x^\star\in\partial D^g$, we define the \textit{inception scaling} $\hat \omega$ as the smallest positive constant such that the streamer integral satisfies relation~\eqref{eq:SIeqKc}, that is,
\begin{equation}\label{eq:defomegahat}
    \hat{\omega} := \min\left\{ \omega>0 : S(\omega, \overline{\E}, \nu, \hat{T})=K_c \right\}.
\end{equation}
Note that $\hat\omega$ depends only on $\overline{\E}$ and $\nu$, i.e., $\hat{\omega}=\hat{\omega}\big(\overline{\E}, \nu(\x^\star, \overline{\E})\big)$. The diagram of Figure~\ref{fig:backwardpbTomega} illustrates both the backwards problem solved to find the value of $\hat T$, and the minimization problem solved to find the value of $\hat\omega$.\\

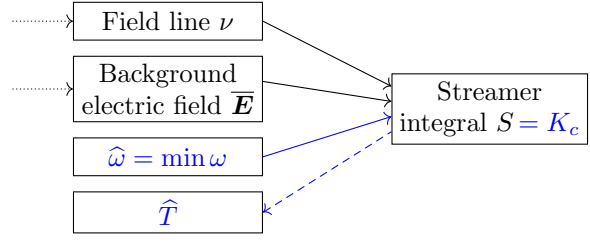
\begin{figure}
    \begin{center}
        \begin{tikzpicture}[
        start chain = going right,
        node distance=7mm,
        block/.style={shape=rectangle, draw,
            inner sep=1mm, align=center,
            minimum height=5mm, minimum width=25mm}] 
        \node[block] (n3r1) {Field line $\nu$};
        \node[block, below=2mm of n3r1] (n2r1) {Background \\electric field $\overline{\E}$};
        \node[block, below=2mm of n2r1] (n3r2) {\color{blue}$\hat\omega=\min \omega$};
        \node[block, below=2mm of n3r2] (n3r3) {\color{blue}$\hat T$};
        \node[block] (n4) at ($(n3r2.east)!0.5!(n3r3.east) + (30mm,10mm)$) {Streamer \\integral $S{\,\color{blue}=K_c}$};

        \draw[->,densely dotted] ($(n3r1.west) + (-0.8cm,0)$) -- (n3r1.west);
        \draw[->,densely dotted] ($(n2r1.west) + (-0.8cm,0)$) -- (n2r1.west);
        \draw[->] (n3r1.east) -- ($(n4.west) + (0,3mm)$);
        \draw[->] ($(n2r1.east) + (0,1mm)$) -- ($(n4.west) + (0,1mm)$);
        \draw[->,blue] (n3r2.east) -- ($(n4.west) + (0,-1mm)$);
        \draw[<-,blue,densely dashed] (n3r3.east) -- ($(n4.west) + (0,-3mm)$);
        \end{tikzpicture}
        \caption{Backward problem: finding $\hat T \in (0, T^c]$ such that the streamer integral $S$ is equal to the streamer constant $K_c$. The minimization problem defining $\hat\omega$ corresponds to finding the minimal value of $\omega$ for which the backward problem has a solution.} \label{fig:backwardpbTomega}
    \end{center}
\end{figure}

Finally, the \textit{inception voltage} $\hat{V}\in\R$ is determined as
\begin{equation} \label{eq:definceptionvoltage}
    \hat{V} := \min_{\x^\star\in \partial D^g} \hat{\omega}\big(\overline{\E}, \nu(\x^\star, \overline{\E})\big).
\end{equation}
Note that $\partial D^g$ corresponds to the gas-solid interfaces. However, in practice, the minimization is only done over the \textit{critical set} defined as
\begin{equation}\label{eq:criticalset}
    X^\star(\omega_u) := \left\{ \x^\star \in \partial D^g :\, |\omega_u\overline{\E}(\x^\star)|>E_c\right\},
\end{equation}
where $\omega_u$ is a user pre-defined positive voltage scaling.

\section{Goal-oriented defeaturing error estimation for inception voltages}\label{s:framework}
In this section, we suppose that we are interested in computing the inception voltage $\hat V$ for a given switchgear design $\Upsilon$, but that, because of software or computational cost restrictions, we can only compute the inception voltage $\hat V_0$ for a simplified switchgear design $\Upsilon_0$ instead. We are therefore interested in estimating the geometry-induced modelling error $|\hat V - \hat V_0|$, also called \emph{geometry simplification error}, without the knowledge of $\hat V$.

In the following, we append the index $0$ to every quantity introduced in Section~\ref{s:models} which relates to the simplified switchgear design $\Upsilon_0$. For instance, the notation $D^g$ will refer to the gas volume of $\Upsilon$, while $D^g_0$ will refer to the gas volume of $\Upsilon_0$.

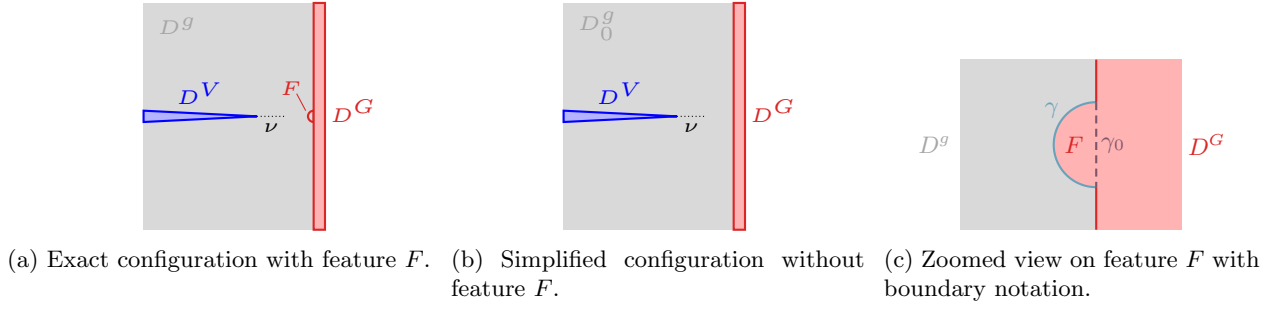
\begin{figure}
\centering
\begin{subfigure}[t]{0.34\textwidth}
\centering
\begin{tikzpicture}[scale=1.5, transform shape, trim left=0cm, trim right=2cm]
\fill[gray!30!white] (0,-1) rectangle (1.5,1);
\node[gray!70!white] at (0.3,0.8) {\tiny$D^g$};
\fill[blue!30] (0,-0.05) -- (0,0.05) -- (1,0) -- cycle;
\draw[thick, blue] (0,-0.05) -- (0,0.05) -- (1,0) -- cycle;
\node[blue] at (0.5,0.2) {\tiny$D^V$};
\draw[thick, red!70!gray] (1.5,0.05) arc[start angle=90, end angle=270, radius=0.05] -- cycle;
\node[red!70!gray] at (1.3,0.25) {\tiny$F$};
\draw[thin,red!70!gray] (1.35,0.2) -- (1.43,0.05);
\fill[red!30] (1.5,-1) rectangle (1.6,1);
\draw[thick, red!70!gray] (1.5,-1) rectangle (1.6,1);
\node[red!70!gray] at (1.85,0.05) {\tiny$D^G$};
\draw[densely dotted] (1,0) -- (1.25,0);
\node at (1.125,-0.1) {\tiny$\nu$};
\end{tikzpicture}
\caption{Exact configuration with feature~$F$.}
\label{fig:sub1}
\end{subfigure}\hfill
\begin{subfigure}[t]{0.33\textwidth}
\centering
\begin{tikzpicture}[scale=1.5, transform shape, trim left=2.5cm, trim right=0cm]
\fill[gray!30!white] (0,-1) rectangle (1.5,1);
\node[gray!70!white] at (0.3,0.8) {\tiny$D^g_0$};
\fill[blue!30] (0,-0.05) -- (0,0.05) -- (1,0) -- cycle;
\draw[thick, blue] (0,-0.05) -- (0,0.05) -- (1,0) -- cycle;
\node[blue] at (0.5,0.2) {\tiny$D^V$};
\fill[red!30] (1.5,-1) rectangle (1.6,1);
\draw[thick, red!70!gray] (1.5,-1) rectangle (1.6,1);
\node[red!70!gray] at (1.85,0.05) {\tiny$D^G$};
\draw[densely dotted] (1,0) -- (1.25,0);
\node at (1.125,-0.1) {\tiny$\nu$};
\end{tikzpicture}
\caption{Simplified configuration without feature~$F$.}
\label{fig:sub2}
\end{subfigure}\hfill
\begin{subfigure}[t]{0.3\textwidth}
\centering
\begin{tikzpicture}[scale=22.5]
\fill[gray!30!white] (1.42,-0.05) rectangle (1.55,0.05);
\fill[red!30] (1.5,0.025) arc[start angle=90, end angle=270, radius=0.025] -- cycle;
\draw[thick, cyan!40!gray] (1.5,0.025) arc[start angle=90, end angle=270, radius=0.025];
\node[red!70!gray] at (1.487,0) {\small$F$};
\node[cyan!40!gray] at (1.474,0.02) {\small$\gamma$};
\fill[red!30] (1.5,-0.05) rectangle (1.55,0.05);
\draw[thick, red!70!gray] (1.5,-0.05) -- (1.5,-0.025);
\draw[thick, red!70!gray] (1.5,0.025) -- (1.5, 0.05);
\draw[thick, densely dashed, magenta!40!black] (1.5,0.025) -- (1.5,-0.025);
\node[magenta!40!black] at (1.51,0) {\small$\gamma_0$};
\node[red!70!gray] at (1.565,0) {\small$D^G$};
\node[gray!70!white] at (1.405,0) {\small$D^g$};
\end{tikzpicture}
\caption{Zoomed view on feature~$F$ with boundary notation.}
\label{fig:notationboundaries}
\end{subfigure}
\caption{Example of a two-electrode configuration satisfying Assumptions~\ref{as:allassumptions}.}
\label{fig:illustrationassumption}
\end{figure}

Our analysis of geometry simplification error estimation relies on the following assumptions:
\begin{assumption}\label{as:allassumptions}
\begin{enumerate}
    \item $\Upsilon$ and $\Upsilon_0$ both correspond to a simple two-electrode configuration falling under Assumption~\ref{as:2elecgeom}. 
    \item $\Upsilon_0$ is obtained from $\Upsilon$ by removing one protrusion $F$, also called \emph{feature}, from one of the electrodes. Equivalently, $D^g_0 = \mathrm{int}\!\left(\overline{D^g} \cup \overline{F}\right)$, and thus in particular $D^g\subset D^g_0$. In the literature on analysis-aware geometry simplification, since the differential problem~\eqref{eq:poissonproblemwithassumptions} is solved in the gas domain $D^g$, $F$ is called \emph{negative} feature. The cases of a \emph{positive} or a \emph{complex} feature go beyond the scope of this manuscript.
    \item The gas volumes $D^g$ and $D^g_0$ are $C^3$-regular (see Section~\ref{ss:modelwithreducingassumptions}).
    \item\label{as:samexstar} When the inception voltage is applied to the two electrodes, inception occurs for both $\Upsilon$ and $\Upsilon_0$ from the same inception point $\x^\star\in\partial D^g$, assumed known \textit{a priori}. Hence, recalling the definition of the inception voltage given by~\eqref{eq:definceptionvoltage},
    \begin{equation}
        |\hat V - \hat V_0| = \left| \hat{\omega}\!\left(\overline{\E}, \nu(\x^\star, \overline{\E})\right) - \hat{\omega}\!\left(\overline{\E}_0, \nu(\x^\star, \overline{\E}_0)\right) \right|.
    \end{equation}
    This assumption is quite restrictive, as it is often but not always verified in practice. If it is not satisfied, the analysis would be more involved since the inception voltage would then be computed along completely different field lines for $\Upsilon$ and $\Upsilon_0$, and the resulting error would also include a contribution from the difference between these two field lines. This case is left for future work.

    \item\label{as:hatomegaambiguity} Feature $F$ does not interact with the critical portion of the $\overline{\E}_0$-field line originating at $\x^\star$, in the following sense. For all $t\in (0, \hat T_0)$, let 
    \begin{equation} \label{eq:tangent}
        \mathbf{t}:=\displaystyle\frac{\overline{\E}_0\big(\x_0(t)\big)}{\left|\overline{\E}_0\big(\x_0(t)\big)\right|}=\frac{\mathrm{d}\x_0}{\mathrm{d}t}(t)
    \end{equation}
    be the unit tangent vector to the field line, let
    \begin{equation}\label{eq:orthogonalplane}
        \Pi_t := \left\{\y\in D^g_0\,:\,(\y-\x_0(t))\cdot\mathbf{t}=0\right\}
    \end{equation}
    be the plane orthogonal to the field line at $\x_0(t)$, and let 
    \begin{equation}
        \Pi := \bigcup_{t\in(0,\hat T_0)} \Pi_t.
    \end{equation}
    Then there is an open neighborhood $\Pi^\mathrm{N}$ of $\Pi$ which does not intersect $F$, i.e., $\Pi^\mathrm{N}\subset D^g\subset D^g_0$.
    This assumption is reasonable: if the critical field line did interact with $F$, the corresponding region would carry a supercritical electric field with strong influence on the inception voltage, suggesting that the feature should not be removed.

    We can then easily remove the ambiguity in the definition of $\hat\omega$ when applied both to $\Upsilon$ and~$\Upsilon_0$ by setting $\hat{\omega}(\overline{\E}_0):=\hat{\omega}\!\left(\overline{\E}_0\vert_{D^g}\right)$, using a slight abuse of notation. In particular, we thus consider $\hat\omega \in \left(\big[L^2(D^g)\big]^3\right)^* \cong \big[L^2(D^g)\big]^3$ (up to isomorphism). Ambiguities in the definitions of $\nu$, $S$ and $\hat T$ are similarly resolved.

    \item\label{as:notdistortedfieldlines} The presence or absence of feature $F$ does not distort the field line originating at $\x^\star$, i.e., 
    \begin{equation}
        \nu(\x^\star, \overline{\E})=\nu(\x^\star, \overline{\E}_0)\cap\overline{D^g}.
    \end{equation} 
    In other words, the background electric field $\overline{\E}_0$ has the same direction as $\overline{\E}$ along $\nu(\x^\star, \overline{\E})$ (but not necessarily the same magnitude). To simplify the notation, we will therefore not write the dependence of $\hat\omega$ and of $\hat T$ (defined in~\eqref{eq:SIeqKc}) on $\nu$, and the inception voltage error reduces to
    \begin{equation}
        |\hat V - \hat V_0| = \left| \hat{\omega}(\overline{\E}) - \hat{\omega}(\overline{\E}_0) \right|.
    \end{equation}
    A relevant case satisfying this assumption is illustrated in Figure~\ref{fig:illustrationassumption} (see also Section~\ref{s:numexample}).

    \item\label{as:highlynonhomog} The equality 
    \begin{equation}
            \hat\omega(\overline{\E}_0)\big|\overline{\E}_0\!\left(\x_0(\hat T_0)\right)\big| = E_c
    \end{equation}
    holds, where $\x_0(t)$ are the points along the field line $\nu_0=\nu$ originated at $\x^\star$, and $\hat T_0$ is defined as
    \begin{equation}\label{eq:T0hat}
        \hat T_0 := \hat T\!\left(\hat\omega(\overline{\E}_0),\overline{\E}_0\right),
    \end{equation}
    so that $\alpha_\mathrm{eff}\!\left(\hat\omega(\overline{\E}_0)\big|\overline{\E}_0\!\left(\x_0(\hat T_0)\right)\big|\right)=0$. This assumption is satisfied, for instance, in two-electrode configurations producing highly inhomogeneous electric fields, as illustrated in Figure~\ref{fig:illustrationassumption}. Remark that from Assumption~\ref{as:allassumptions}.\ref{as:hatomegaambiguity}, $\hat T_0 = \hat T\!\left(\hat\omega(\overline{\E}_0\vert_{D^g}),\overline{\E}_0\vert_{D^g}\right)$; thus $\hat T_0$ is a functional of the electric field restricted to $D^g$, that is, $\hat T_0\in \left(\big[L^2(D^g)\big]^3\right)^* \cong \big[L^2(D^g)\big]^3$ (up to isomorphism).\\
\end{enumerate}
\end{assumption}

To the best of our knowledge, the literature provides analysis-aware defeaturing estimators for:
\begin{itemize}
    \item energy errors, see~\cite{paper1defeaturing, weder2025analysis, paper3multifeature, phdthesis, buffa2024equilibrated, buffa2022adaptive}, i.e., for our application, $\left\|\nabla(\overline{\varphi} - \overline{\varphi}_0\vert_{D^g})\right\|_{0,D^g} = \left\|\overline{\E} - \overline{\E}_0\vert_{D^g}\right\|_{0,D^g}$;
    \item linear quantities of interest (QoIs), see~\cite{weder2025certified}, i.e., for our application,
    $\left| J(\overline{\varphi}) - J(\overline{\varphi}_0\vert_{D^g})\right|$, where $J\in H^{-1}(D^g)$ is a linear functional of the background electric potential.
\end{itemize}
Energy is a spatially averaged quantity and does not provide a reliable criterion for predicting dielectric breakdown, which is governed by localized field concentrations and singularities. In contrast, QoIs can capture such local effects; however, the inception voltage~$\hat\omega$ is a highly nonlinear functional of the electric field. 

Our strategy is as follows. We first recall in Section~\ref{ss:estimatorfrom1storder} the certified goal-oriented analysis-aware defeaturing estimator of~\cite{weder2025certified}, which applies to any regular linear QoI belonging to $H^{-1}$. In Section~\ref{ss:1storderapprox} we derive a linear first-order approximation $J$ of $\hat\omega$. The QoI $J$ obtained in this way is, however, only of low regularity ($J\in H^{-s}$, $s>2$), and its direct use in the estimator of Section~\ref{ss:estimatorfrom1storder} is not licit: Section~\ref{ss:regularization} addresses this issue by introducing a Gaussian regularization $J_r$ of $J$ that recovers the required $H^{-1}$-regularity. Finally, Section~\ref{ss:puttingittogether} combines all the ingredients into a computable error bound for the inception voltage.

\subsection{Goal-oriented defeaturing estimator for regular quantities of interest}\label{ss:estimatorfrom1storder}

We recall the certified analysis-aware defeaturing estimator from~\cite{weder2025analysis,weder2025certified} adapted to our application. Let
\begin{equation}\label{eq:defgammagamma0}
    \gamma:=\mathrm{int}\!\left(\overline{\partial D^g} \cap \overline{\partial F}\right) \quad \text{ and } \quad \gamma_0:=\partial F\setminus\overline{\gamma}.
\end{equation}
That is, $\gamma$ is the \emph{defeatured} portion of $\partial D^g$ which is removed when $F$ is removed, and $\gamma_0$ is the \emph{simplified} portion of $\partial D^g_0$ which replaces $\gamma$. In particular, $\partial F=\mathrm{int}(\gamma \cup \gamma_0)$. This notation is illustrated in Figure~\ref{fig:notationboundaries}. 
Note from~\eqref{eq:poissonproblemwithassumptions} that the boundary conditions imposed on both $\gamma$ and $\gamma_0$ are of the same Dirichlet type. Let us denote by $V_D\in\{0,1\}$ this Dirichlet value imposed on $\gamma$ and $\gamma_0$.\\

The main results from~\cite{weder2025analysis,weder2025certified}, applied to our setting, then read as follows.

\begin{theorem} \label{th:wederbuffa}
Let $\mathcal E_{H^1}\colon H^1(D_0^g)\times \R\to\R$ be the functional
\begin{equation}\label{eq:H1functional}
    \mathcal E_{H^1}(u_0,V) := 2\sqrt{\|V-u_0\|_{0,\gamma}\,\|\nabla_{\mathbf{t}}(V-u_0)\|_{0,\gamma}},\quad\text{for all $u_0\in H^1(D_0^g)$ and $V\in\R$},
\end{equation}
where $\nabla_{\mathbf{t}}$ denotes the tangential gradient along $\gamma$. Let furthermore $J^{\,\mathrm{reg}}\in H^{-1}(D^g)$ be a regular linear quantity of interest, and define $J^{\,\mathrm{reg}}_0\in H^{-1}(D_0^g)$ as 
\begin{equation}
    J^{\,\mathrm{reg}}_0(u_0) := J^{\,\mathrm{reg}}(u_0\vert_{D^g}) \quad \text{for all $u_0\in H^1(D_0^g)$},
\end{equation}
and let $z_0^{\mathrm{reg}}\in H^1_0(D_0^g)$ be the corresponding so-called \emph{dual} solution, satisfying
\begin{equation}\label{eq:dual}
    \begin{cases}
        -\Delta z_0^{\mathrm{reg}} = J^{\,\mathrm{reg}}_0 & \text{ in }D_0^g,\\
        z_0^{\mathrm{reg}} = 0 & \text{ on } \partial D_0^g.
    \end{cases}
\end{equation}
In the weak sense, the problem reads: find $z_0^{\mathrm{reg}}\in H^1_0(D_0^g)$ such that, for all $z_t\in H^1_0(D_0^g)$,
\begin{equation}\label{eq:dualweak}
    \int_{D_0^g} \nabla z_0^{\mathrm{reg}}\cdot\nabla z_t = J^{\,\mathrm{reg}}_0\left(z_t\right).
\end{equation}
Denoting by $\mathbf{n}$ the unit normal vector on $\gamma$ pointing outward from $D^g$, the following estimates hold:
\begin{align}
    &\left\| \overline{\E}-\overline{\E}_0\vert_{D^g} \right\|_{0,D^g} \leq C_1\, \mathcal E_{H^1}(\overline{\varphi}_0,V_D), \label{eq:H1est}\\
    &\left| J^{\,\mathrm{reg}}(\overline{\varphi}) - J^{\,\mathrm{reg}}(\overline{\varphi}_0\vert_{D^g}) + \int_{\gamma} \nabla z_0^{\mathrm{reg}}\cdot\mathbf{n}\left(V_D-\overline{\varphi}_0\right)\,\mathrm ds \right| \leq C_2\,\mathcal E_{H^1}(\overline{\varphi}_0, V_D)\,\mathcal E_{H^1}(z_0^{\mathrm{reg}}, 0), \label{eq:GOest}
\end{align}
where $C_1$ and $C_2$ are constants independent of the feature size but possibly depending on its shape.
\end{theorem}

\begin{remark} We note that:
    \begin{itemize}
        \item The two estimators on the right-hand sides of~\eqref{eq:H1est} and~\eqref{eq:GOest} capture the correct scaling with respect to the feature size: as $|F|$ shrinks, both the geometry simplification errors and the estimators decrease at the same rate, irrespective of $C_1$ and $C_2$. Instead, an important open question concerns how the feature's shape affects the constants $C_1$ and $C_2$, since features with high curvature induce stronger local field perturbations than gently rounded features of the same size. A rigorous investigation of this shape-dependence would thus be further needed, but some insights can already be found in the numerical experiments conducted in Section~\ref{s:numexample}.
        \item The estimators only require boundary integrals over the defeatured surface $\gamma$, and they are computable from the simplified primal and dual solutions only. This also means that the whole error (in the energy norm or in some QoI) is fully captured by the behavior of the solutions on $\gamma$.
        \item Bound~\eqref{eq:H1est} tells us that $\mathcal{E}_{H^1}$ captures the geometry simplification error in the energy norm. The dual problem~\eqref{eq:dual} is constructed to localize the QoI (for our application, an approximation of the inception voltage), indicating where it is most sensitive. That is, the product $\mathcal E_{H^1}(\overline{\varphi}_0, V_D)\mathcal E_{H^1}(z_0, 0)$ captures the global averaged geometry simplification error, weighted by the sensitivity to the QoI.
        \item The corrector term
        \begin{equation} \label{eq:corrector}
            R(z_0^{\mathrm{reg}}, \overline{\varphi}_0, V_D) := -\int_{\gamma} \nabla z_0^{\mathrm{reg}}\cdot\mathbf{n}\left( V_D-\overline{\varphi}_0 \right) \,\mathrm ds
        \end{equation}
        provides a first-order correction of the defeatured quantity $J^{\,\mathrm{reg}}(\overline{\varphi}_0\vert_{D^g})$, and the approximation
        \begin{equation*}
            J^{\,\mathrm{reg}}(\overline{\varphi}) \approx J^{\,\mathrm{reg}}(\overline{\varphi}_0\vert_{D^g}) + R(z_0^{\mathrm{reg}}, \overline{\varphi}_0, V_D)
        \end{equation*}
        has accuracy controlled by the product $\mathcal E_{H^1}(\overline{\varphi}_0, V_D)\,\mathcal E_{H^1}(z_0^{\mathrm{reg}}, 0)$.
        \item For Theorem~\ref{th:wederbuffa} to be useful in practice, the dual problem~\eqref{eq:dualweak} with $J_0^\mathrm{reg}$ on the right-hand side should not depend explicitely on the exact domain $D^g$ (for instance, it should not be the solution's average over $D^g$). Indeed, geometry defeaturing aims to avoid solving problems on the exact domain, or more precisely, when using a finite element solver, to avoid meshing it. 
    \end{itemize}
\end{remark}

\subsection{First-order approximation of the inception voltage error}\label{ss:1storderapprox}
\footnotetext[8]{More precisely, since $\hat\omega$ is a functional, $\hat\omega'(\overline{\E}_0)\cdot\left(\overline{\E}-\overline{\E}_0\vert_{D^g}\right)$ is the G\^ateaux derivative of $\hat\omega$ in the direction $\overline{\E}-\overline{\E}_0\vert_{D^g}$, evaluated at $\overline{\E}_0\vert_{D^g}$.}
By first-order Taylor expansion,
\begin{equation}\label{eq:firstorderapprox}
    \hat{\omega}(\overline{\E}) \approx \hat{\omega}(\overline{\E}_0) + \hat\omega'(\overline{\E}_0)\cdot(\overline{\E}-\overline{\E}_0\vert_{D^g}) = \hat{\omega}(\overline{\E}_0) - \hat\omega'(\overline{\E}_0)\cdot\nabla(\overline{\varphi} - \overline{\varphi}_0\vert_{D^g}),
\end{equation}
where $\hat\omega'(\overline{\E}_0)$ denotes the derivative\footnotemark[8] of $\hat\omega$ evaluated in $\overline{\E}_0$. If regular enough, the functional~$J$ given by
\begin{equation}\label{eq:functionalJ}
    J(\delta\varphi) := - \hat\omega'(\overline{\E}_0)\cdot \nabla\delta\varphi
    \quad \text{for all electric potentials $\delta\varphi\in H^1(D^g)$ such that }\nu(\x^\star,\nabla\delta\varphi)=\nu(\x^\star,\overline{\E}_0)\cap\overline{D^g},
\end{equation}
would be a good linear QoI candidate, varying with the background electric field similarly to the inception voltage. The next proposition gives a closed-form expression of this derivative evaluated in the electric potential's defeaturing error; the proof is found in Appendix~\ref{appendix} for completeness.

\begin{proposition}\label{prop:1storderapprox}
Let $\delta\overline{\varphi}:=\overline{\varphi}-\overline{\varphi}_0\vert_{D^g}$, and assume that $\hat T_0$ from~\eqref{eq:T0hat} is differentiable at $\overline{\E}_0$ in the direction~$\nabla\delta\overline{\varphi}$. Then under Assumption~\ref{as:allassumptions}, the functional $J$ defined by~\eqref{eq:functionalJ} admits the following closed-form expression when evaluated in $\delta\overline{\varphi}$:
\begin{equation}\label{eq:Jclosedform}
    J(\delta\overline{\varphi}) = \frac{\hat\omega(\overline{\E}_0) }{Z(\overline{\E}_0)}
    \int_0^{\hat T_0} \kappa(t)\mathbf{t}\cdot \nabla\delta\overline{\varphi}(\x_0(t))\,\mathrm dt,
\end{equation}
where 
\begin{align*}
&\kappa(t):= \alpha_\mathrm{eff}'\!\left(\hat\omega(\overline{\E}_0)|\overline{\E}_0(\x_0(t))|\right),\\
& Z(\overline{\E}_0) := \int_0^{\hat T_0} \alpha_\mathrm{eff}'\!\left(\hat\omega(\overline{\E}_0)|\overline{\E}_0(\x_0(t))|\right) |\overline{\E}_0(\x_0(t))|\,\mathrm dt, 
\end{align*}
and we recall from~\eqref{eq:tangent} that $\mathbf{t}$ is the unit tangent vector to the field line.
\end{proposition}

\subsection{Lack of regularity of the goal functional and Gaussian mollification}\label{ss:regularization}

The estimator of Theorem~\ref{th:wederbuffa} requires the QoI $J^\mathrm{reg}$ to belong to the dual space $H^{-1}(D^g)$ so that the dual problem~\eqref{eq:dualweak} is well-posed in $H^1_0(D_0^g)$. This regularity requirement is standard in goal-oriented \textit{a posteriori} error estimators based on the dual-weighted residual paradigm, see~\cite{becker2001optimal,bangerth2003adaptive,giles2002adjoint,prudhomme1999goal}. Unfortunately, the closed-form expression~\eqref{eq:Jclosedform} of $J$ does not satisfy this requirement for general $\delta\varphi$:

\begin{proposition}\label{prop:regularityJ}
    The functional $J$ defined by~\eqref{eq:Jclosedform} belongs to $H^{-2-\epsilon}(D^g)\not\subset H^{-1}(D^g)$, for any $\epsilon>0$.
\end{proposition}
Note that since $\overline{\varphi}$ and $\overline{\varphi}_0$ are solutions of Poisson problem~\eqref{eq:poissonproblemwithassumptions} on $D^g$ and $D^g_0$, respectively, and since $D^g$ and $D^g_0$ are $C^3$-regular, then $\delta\overline{\varphi}\in H^3(D^g)$. Thus $J$ is well-defined on $\delta\overline{\varphi}$, but it does not belong to $H^{-1}(D^g)$ in general. The proof of Proposition~\ref{prop:regularityJ} can be found in Appendix~\ref{appendix} for completeness.

This is a classical lack-of-regularity issue for QoIs supported on lower-dimensional manifolds; see~\cite{adams2003sobolev,lions1972nonhomogeneous,evans2010partial} for the underlying trace and Sobolev-space theory, and~\cite{tornberg2004numerical} for a discussion in the context of singular source terms in PDEs. Plugging~\eqref{eq:Jclosedform} directly into~\eqref{eq:dualweak} would yield a dual solution that does not belong to $H^1_0(D_0^g)$, and the estimate~\eqref{eq:GOest} of Theorem~\ref{th:wederbuffa} would not be valid.\\

To recover $H^{-1}$-regularity while keeping the QoI close to its physical meaning, we replace the line integral in~\eqref{eq:Jclosedform} by a volume integral against the following Gaussian mollifier $\rho_r$ for $r>0$:
\begin{equation}\label{eq:rhoeps}
    \rho_r(\x) := \frac{1}{2\pi r^2}\,\exp\!\left(-\frac{|\x|^2}{2r^2}\right), \quad \text{for all $\x\in D^g$}.
\end{equation}
That is, recalling the definition of $\Pi^\mathrm{N}$ and $\Pi_t$ for $t\in(0,\hat T_0)$ from Assumption~\ref{as:allassumptions}.\ref{as:hatomegaambiguity}, and if we let $\Pi^\mathrm{N}_t:= \Pi^\mathrm{N}\cap\Pi_t$, then the regularized QoI $J_r$ is defined as follows:
\begin{equation}\label{eq:JeRdef}
    J_r(\delta\varphi) := \frac{\hat\omega(\overline{\E}_0)}{Z(\overline{\E}_0)} \int_0^{\hat T_0} \left[ \int_{\Pi^\mathrm{N}_t} \rho_r\big(\y-\x_0(t)\big)\kappa(t)\mathbf{t}\cdot \nabla\delta\varphi(\y) \,\mathrm d\y \right]\,\mathrm dt.
\end{equation}

We can now prove that $J_r$ satisfies the regularity requirement of Theorem~\ref{th:wederbuffa}.
\begin{proposition}\label{prop:regularJregularity}
    The regularized functional $J_r$ defined by~\eqref{eq:JeRdef} belongs to $H^{-1}(D^g)$ for all $r>0$. Moreover, 
    \begin{equation}
        \lim_{r\to 0} J_r = J \quad \text{in the sense of distributions}.
    \end{equation}
\end{proposition}
\begin{proof}
    Given the assumptions on the effective ionization coefficient $\alpha_\mathrm{eff}$ from Section~\ref{ss:streamerintegralmodel}, and the regularity of $\overline{\E}_0$, then $\kappa(t)$ is bounded on $(0,\hat T_0)$. The regularity of $J_r$ then follows from the $C^\infty$-regularity of $\rho_r$.
    
    Moreover, the inner integral in~\eqref{eq:JeRdef} is a convolution of $\kappa(t)\mathbf{t}\cdot \nabla\delta\varphi$ with the Gaussian mollifier~$\rho_r$ on the plane $\Pi^\mathrm{N}_t$, in which $\kappa(t)$ is taken constant. Then it is well-known that in this case, $J_r$ converges to $J$ as $r$ tends to zero (in the sense of distributions), see e.g., Section~C.4 of~\cite{evans2010partial}. 
\end{proof}

\subsection{Final inception voltage defeaturing error estimation}\label{ss:puttingittogether}
We now have all the ingredients to derive a computable error bound for the inception voltage error induced by defeaturing. 

\begin{proposition}\label{prop:finalresult}
    Applying Theorem~\ref{th:wederbuffa} with $J^{\,\mathrm{reg}} = J_r$ yields a well-defined regularized dual solution $z_{0,r}\in H^1_0(D_0^g)$. Let us denote by $\mathcal E_J(z_{0,r}, \overline{\varphi}_0, V_D)$ the following goal-oriented estimator associated with the regularized functional $J_r$:
    \begin{equation}\label{eq:goaldefestimator}
        \mathcal E_J(z_{0,r}, \overline{\varphi}_0, V_D) := \left| \mathcal E_{H^1}(\overline{\varphi}_0, V_D)\,\mathcal E_{H^1}(z_{0,r}, 0) - R(z_{0,r}, \overline{\varphi}_0, V_D)\right|.
    \end{equation}
    Then under Assumption~\ref{as:allassumptions}, the inception voltage error induced by removing feature $F$ can be estimated as follows:
    \begin{equation}
        \left| \hat V - \hat V_0 \right| \approx \left| J(\overline{\varphi})-J(\overline{\varphi}_0\vert_{D^g}) \right|
        \leq \widetilde C_2\,\lim_{r\to 0} \mathcal E_J(z_{0,r}, \overline{\varphi}_0, V_D),
    \end{equation}
    with $\widetilde C_2 = \max(1,C_2)$. 
\end{proposition}
\begin{proof}
    Combining the first-order Taylor expansion~\eqref{eq:firstorderapprox}, Proposition~\ref{prop:1storderapprox}, the Gaussian mollification from Section~\ref{ss:regularization}, and the goal-oriented estimate~\eqref{eq:GOest} of Theorem~\ref{th:wederbuffa} applied to the regularized functional $J_r$, we obtain the following error estimation for the inception voltage error induced by removing feature $F$:
    \begin{equation}\label{eq:errordecomp}
        \begin{aligned}
            \left| \hat V - \hat V_0 \right|
            = \left| \hat\omega(\overline{\E})-\hat\omega(\overline{\E}_0) \right|
            \approx \left| J(\overline{\varphi})-J(\overline{\varphi}_0\vert_{D^g}) \right|
            &= \left| \lim_{r\to 0} \left[ J_r(\overline{\varphi})-J_r(\overline{\varphi}_0\vert_{D^g}) \right] \right|\\
            &= \lim_{r\to 0} \left| J_r(\overline{\varphi})-J_r(\overline{\varphi}_0\vert_{D^g}) \right|\\
            &\leq \widetilde C_2\,\lim_{r\to 0} \mathcal E_J(z_{0,r}, \overline{\varphi}_0, V_D).
        \end{aligned}
    \end{equation}
\end{proof}

Note that the error bound of Proposition~\ref{prop:finalresult} is computable since it only depends on the simplified primal and dual solutions, which are both defined on the simplified domain $D^g_0$, and the right hand side of the dual problem~\eqref{eq:dualweak} is defined on $\Pi^\mathrm{N}$ which is independent of the exact domain $D^g$.
Moreover, the corrector term $R(z_{0,r}, \overline{\varphi}_0, V_D)$ provides a first-order correction of the regularized defeatured quantity $J_r(\overline{\varphi}_0\vert_{D^g})$, and thus of the inception voltage $\hat V_0$.

\begin{remark}\label{rk:regerror} 
    In practice, the limit $r\to 0$ cannot be taken, and the regularization parameter $r$ should be chosen small enough to ensure that $J_r$ is close to $J$, but not too small to avoid numerical issues. In particular, the regularization error $\left| J(\varphi)-J_r(\varphi) \right|$ should thus be controlled to balance with the geometry simplification error captured by $\mathcal E_J(z_{0,r}, \overline{\varphi}_0, V_D)$.
\end{remark}

\section{Numerical illustration on a $2$D pin--plate configuration}\label{s:numexample}
This section illustrates the theory developed in Sections~\ref{s:models} and~\ref{s:framework} on a simple use case. We consider a $2$D pin-plate two-electrode configurations: a $1$~mm-thick, $3$~mm-tall high-voltage pin electrode is faced by a flat grounded plate at a $7$~mm gap, and a small protrusion~$F$ of width~$W$ and height~$H$ is placed on the plate. The protrusion is removed to obtain the simplified design. To ensure $C^\infty$-smoothness and avoid geometric point singularities, the pin and feature shapes are defined by cosine functions. $H$ corresponds to the feature's cosine amplitude and~$W$ to its period. Figure~\ref{fig:pinplate} shows geometries along with the corresponding background electric potentials and fields, obtained from~\eqref{eq:poissonproblemwithassumptions} and~\eqref{eq:defbackgroundelectricfield}, respectively. Note the distortion of the electric field due to feature $F$: this geometric configuration is a common scenario in high-voltage equipment where small geometric features may lead to significant local electric field enhancements and thus to significant errors in inception voltage predictions if not properly accounted for.

\input{figure0-pinplate}

The corresponding three-dimensional geometry is generated by extrusion in the out-of-plane direction. However, the third dimension does not add any particular value to the analysis, and the 2D setting is sufficient to illustrate the methodology. Moreover, the 2D setting allows for a more refined mesh and therefore for a more accurate reference solution, which is critical to verify the reliability of the computed effectivity index. 

In this case, the critical field line $\nu$ is the horizontal segment $y=5$~mm intersected with the gas domain, and the inception point $\x^\star=(3,5)$~mm is the pin tip in both~$\Upsilon$ and~$\Upsilon_0$. The considered geometric configuration satisfies Assumption~\ref{as:allassumptions}. In particular, since the field is highly inhomogeneous, the critical part of the field line does not interact with the feature since $F$ is small enough, and one can choose $\Pi^N$ for Assumption~\ref{as:allassumptions}.\ref{as:hatomegaambiguity} to be any tubular region around the critical part of the field line.


Moreover, by invariance of all quantities with respect to the out-of-plane direction, the regularized QoI $J_r$ defined in~\eqref{eq:JeRdef} can be computed using the 1D Gaussian mollifier $\rho_r^\mathrm{1D} (x,y) = \displaystyle\frac{1}{\sqrt{2\pi}r}\exp\left(-\displaystyle\frac{x^2+y^2}{2r^2}\right)$
in place of $\rho_r$ defined in~\eqref{eq:rhoeps} as follows:
\begin{equation}
    J_r(\delta\varphi) = \frac{\hat\omega(\overline\E_0)}{Z(\overline\E_0)}\int_0^{\hat T_0} \left[\int_{\Pi_t^N} \rho_r^\mathrm{1D}\big(\y-\x_0(t)\big) \kappa(t)\mathbf{t}\cdot \delta\varphi(\y)\,\mathrm{d}S(\y)\right]\,\mathrm{d}t.
\end{equation}

The objective of this section is to assess the reliability of the goal-oriented defeaturing estimator $\mathcal{E}_J(z_{0,r}, \overline{\varphi}_0, V_D)$ derived in Proposition~\ref{prop:finalresult} for different values of~$H$ and~$W$, and in particular to verify that
\begin{enumerate}[label=(\roman*)]
    \item the estimator $\mathcal{E}_J$ correctly captures the magnitude and the scaling of the inception voltage error~$|\hat V-\hat V_0|$ with respect to the feature's size;
    \item the resulting effectivity index $\mathcal{I}_\mathrm{eff}:= \displaystyle\frac{\mathcal{E}_J}{|\hat V-\hat V_0|}$ is a property of the geometry simplification alone, i.e., it is not polluted by either the discretization error or the regularization error introduced in Section~\ref{ss:regularization}.
\end{enumerate}
To lighten the notation since $V_D=0$, we remove the explicit dependence on $V_D$ from now on. For instance, we write $\mathcal{E}_J(z_{0,r}, \overline{\varphi}_0)$ instead of $\mathcal{E}_J(z_{0,r}, \overline{\varphi}_0, 0)$.

Achieving (ii) is non-trivial: the dual problem~\eqref{eq:dual} for the regularized QoI~$J_r$ introduced in Section~\ref{ss:regularization} has a Gaussian-tube source concentrated along the critical part of the field line~$\nu$, and the dual solution $z_{0,r}$ exhibits a sharp transverse layer. Therefore, a careful convergence study is required to ensure that the defeaturing error dominates over both the discretization and the regularization errors. \\

We use the finite element method on the gas domains $D^g$ and $D^g_0=\mathrm{int}(\overline{D^g}\cup\overline{F})$, truncated above and below by artificial boundaries $y=0$ and $y=10$~mm, on which homogeneous Neumann boundary conditions are imposed. The meshes are generated using Netgen, see~\cite{schoberl1997netgen}, via the NGSolve open-source finite element library, see~\cite{ngsolve}; first-order Lagrange elements are used everywhere. The meshes on $D^g$ and $D^g_0$ share the same nodes on the boundary~$\gamma$ (i.e., the meshes are conformal across~$\gamma$), so that the boundary integrals appearing in~\eqref{eq:H1functional} and~\eqref{eq:corrector} can be computed without any quadrature inconsistency. This is not necessary but it reduces the discretization error, whose control in combination with the defeaturing error goes beyond the scope of this manuscript. Furthermore, the pin and feature boundaries are discretized with curved edges of length $\displaystyle\frac{10}{2^{10}}$~mm. The finite element solutions of~\eqref{eq:poissonproblemwithassumptions} in $D^g$ and $D^g_0$ are denoted $\overline{\varphi}_h$ and $\overline{\varphi}_{0,h}$, respectively. More generally, we append an $h$ when referring to any discretized quantity.

The streamer integral~$S$ from~\eqref{eq:defS}, the inception voltages $\hat\omega(\overline{\E}_h)$ and $\hat\omega(\overline{\E}_{0,h})$, and the linear QoI $J$ from~\eqref{eq:Jclosedform} are computed along the known field line~$\nu=\{\x\in D^g: y=5\,\mathrm{mm}\}$, using the algorithm described in~\cite[Sec.~5.2]{thesis:xeno}. Rather than solving~\eqref{eq:deffieldlines} numerically, we use the known field line to avoid introducing additional numerical errors. Since first-order finite elements are used, $\overline{\E}_h$ and $\overline{\E}_{0,h}$ are piecewise constant and the line integrals are evaluated edge-by-edge. The dual solution $z_{0,r}$ is obtained by solving the regularized dual problem~\eqref{eq:dualweak} with $J^{\,\mathrm{reg}}=J_r$ as defined in~\eqref{eq:JeRdef}.\\

To achieve (ii), we combine the following ingredients:
\begin{itemize}
    \item an initial mesh of elements' diameter approximately $0.5$~mm which is manually refined to have elements of diameter approximately $10^{-3}$~mm near the electrode boundary, the feature boundary, and the line $y=5$~mm;
    \item a single shared adaptive mesh refinement strategy per feature's parameters~$(W,H)$, for which refinement is performed according to the union of the markings produced by classical primal and dual residual-based $H^1$-numerical error estimators, see for instance~\cite{babuvska1978posteriori,verfurth1996review}, and a DWR (dual weighted residual) goal-oriented error estimator corresponding to $J_r(\overline{\varphi}-\overline{\varphi}_0\vert_{D^g})$, see for instance~\cite{becker2001optimal}. In this way, all quantities entering~\eqref{eq:goaldefestimator} are computed on the same mesh. We refer to Section~\ref{ss:adaptive} for more details;
    \item a coupling $r = 2 h_{\nu}$ between the regularization parameter $r$ and the local mesh size $h_\nu$, where $h_\nu$ is the diameter of the largest cell intercepting the critical part of the field line $\nu$ (see Remark~\ref{rk:regerror}); the multiplicative factor~$2$ is used to ensure that at least two elements cover the Gaussian tube in the $y$-direction;
    \item the decomposition
    \begin{equation}\label{eq:decompositionerrorcontrol}
        J(\delta\overline{\varphi}) = J_r(\delta\overline{\varphi}_h) + J_r(\delta\overline{\varphi}-\delta\overline{\varphi}_h) + \left(J(\delta\overline{\varphi}_h) - J_r(\delta\overline{\varphi}_h)\right) + \left(J(\delta\overline{\varphi}-\delta\overline{\varphi}_h) - J_r(\delta\overline{\varphi}-\delta\overline{\varphi}_h)\right),
    \end{equation}
    where $\delta\overline{\varphi}:=\overline{\varphi}-\overline{\varphi}_0\vert_{D^g}$ and $\delta\overline{\varphi}_h:=\overline{\varphi}_h-\overline{\varphi}_{0,h}\vert_{D^g}$. 
    \begin{itemize}
        \item The first term $J_r(\delta\overline{\varphi}_h)$ is the computable quantity corresponding to the geometry simplification error of interest, which however includes numerical and regularization errors.
        \item The second term $J_r(\delta\overline{\varphi}-\delta\overline{\varphi}_h)$ is the goal oriented numerical error evaluated in the defeaturing error.
        \item The third term $J(\delta\overline{\varphi}_h) - J_r(\delta\overline{\varphi}_h)$ corresponds to the regularization error evaluated in the discrete defeaturing error.
        \item The last term $J(\delta\overline{\varphi}-\delta\overline{\varphi}_h) - J_r(\delta\overline{\varphi}-\delta\overline{\varphi}_h)$ corresponds to a higher-order term combining the regularization error and the numerical error evaluated in the defeaturing error.
    \end{itemize}
    By simultaneously monitoring all the first-order contributions along the adaptive iterations, we can ensure that the computed effectivity index reflects only the geometry simplification error.
\end{itemize}

\subsection{Combined adaptive single-mesh convergence study}\label{ss:adaptive}
For each pair of feature parameters~($W$,$H$), we run a single combined refinement strategy to achieve~(ii) and to ensure that all the quantities entering in the estimation~\eqref{eq:errordecomp} are computed on the same mesh and converge simultaneously. The use of a single mesh removes the need to transfer functions (and the associated transfer error) between primal, dual and goal-oriented meshes and between the exact and simplified geometries. The refinement strategy also allows us to separate the geometry simplification error, which depends only on~$W$ and $H$, from the discretization and regularization errors, which both vanish as the mesh is refined. 

The combined adaptive mesh refinement loop consists in iteratively executing the following steps:
\begin{enumerate}[label=(\alph*)]
    \item solve the primal problems~\eqref{eq:poissonproblemwithassumptions} on $D^g$ and $D^g_0$ to obtain $\overline{\varphi}_h$ and $\overline{\varphi}_{0,h}$, respectively, and compute the corresponding background electric fields $\overline{\E}_h$ and $\overline{\E}_{0,h}$;
    \item compute the inception voltages $\hat V_h$ and $\hat V_{0,h}$ along~$\nu$, the linear QoI~$J$ given by~\eqref{eq:Jclosedform}, and its regularization~$J_r$ given by~\eqref{eq:JeRdef} with $r=2 h_{\nu}$;
    \item solve the regularized dual problems~\eqref{eq:dualweak} on $D^g$ and $D^g_0$ to obtain $z_{r,h}$ and $z_{0,r,h}$, respectively;
    \item\label{item:h1num} evaluate the residual-based $H^1$-numerical error estimator $\mathcal{E}_N$ from~\cite{verfurth1996review} for the primal and the dual solutions, and for both the exact and simplified geometries. That is, we obtain the following four estimates: 
        \begin{itemize}
            \item $\mathcal{E}_{N}(\overline{\varphi}_h)$ estimating the numerical error $\left\|\overline{\varphi}-\overline{\varphi}_h\right\|_{H^1(D^g)}$, 
            \item $\mathcal{E}_{N}(\overline{\varphi}_{0,h})$ estimating the numerical error $\left\|\overline{\varphi}_0-\overline{\varphi}_{0,h}\right\|_{H^1(D^g_0)}$, 
            \item $\mathcal{E}_{N}(z_{r,h})$ estimating the numerical error $\left\|z_r-z_{r,h}\right\|_{H^1(D^g)}$,
            \item $\mathcal{E}_{N}(z_{0,r,h})$ estimating the numerical error $\left\|z_{0,r}-z_{0,r,h}\right\|_{H^1(D^g_0)}$;
        \end{itemize}
    \item\label{item:dwrnum} evaluate the quantity $\mathcal E_{\mathrm{DWR}}(\delta\overline{\varphi}_h)$ corresponding to the residual-based DWR goal-oriented estimator of the numerical error $\left|J_r(\overline{\varphi}-\overline{\varphi}_{0}\vert_{D^g})-J_r(\overline{\varphi}_h-\overline{\varphi}_{0,h}\vert_{D^g})\right|=\left|J_r(\delta\overline{\varphi}-\delta\overline{\varphi}_h)\right|$ from~\cite{becker2001optimal}. To do so, we use an enriched dual finite element space by keeping the same mesh on $D^g$ but globally raising the polynomial order of the finite elements from $1$ to $2$, and we use the corresponding enriched dual solution to evaluate the DWR estimator;
    \item mark elements using D\"orfler's strategy with parameter $\theta=0.5$, see~\cite{dorfler1996convergent}, for each of the five estimators from steps~\ref{item:h1num} and~\ref{item:dwrnum} separately. We obtain the sets of marked elements denoted $\mathcal M_{\overline{\varphi}}$, $\mathcal M_{\overline{\varphi}_0}$, $\mathcal M_z$, $\mathcal M_{z_0}$ and $\mathcal M_{\mathrm{DWR}}$, respectively;
    \item using $\overline{\varphi}_{0,h}$ and $z_{0,r,h}$, evaluate the defeaturing error estimators~\eqref{eq:H1est} in energy norm for the primal and dual problems, the corrector term $R\!\left(z_{0,r,h},\overline{\varphi}_{0,h},0\right)$ from~\eqref{eq:corrector}, and the defeaturing error estimator~\eqref{eq:goaldefestimator} for the QoI~$J_r$. That is, we obtain the following three estimates:
        \begin{itemize}
            \item $\mathcal{E}_{H^1}(\overline{\varphi}_{0,h})$ estimating the primal defeaturing error $\left\|\overline{\E}-\overline{\E}_{0}\vert_{D^g}\right\|_{0,D^g}$ with numerical errors,
            \item $\mathcal{E}_{H^1}(z_{0,r,h})$ estimating the dual defeaturing error $\left\|z_{r}-z_{0,r}\vert_{D^g}\right\|_{H^1(D^g)}$ with numerical errors,
            \item $\mathcal{E}_J(z_{0,r,h},\overline{\varphi}_{0,h})$ estimating the goal-oriented defeaturing error $|J_r(\overline{\varphi})-J_r(\overline{\varphi}_{0}\vert_{D^g})|$ with numerical errors;
        \end{itemize} 
    \item evaluate the regularization error $E_r:=\left|J(\delta\overline{\varphi}_h)-J_r(\delta\overline{\varphi}_h)\right|$;
    \item\label{item:markingconditions} define the union of marked elements denoted $\mathcal M$ as follows:
        \begin{itemize}
            \item if $\mathcal{E}_{N}(\overline{\varphi}_h) > 0.5 \mathcal{E}_{H^1}(\overline{\varphi}_{0,h})$ or $\mathcal{E}_{N}(\overline{\varphi}_{0,h}) > 0.5 \mathcal{E}_{H^1}(\overline{\varphi}_{0,h})$, add $\mathcal M_{\overline{\varphi}}\cup\mathcal M_{\overline{\varphi}_0}$ to $\mathcal M$,
            \item if $\mathcal{E}_{N}(z_{r,h}) > 0.5 \mathcal{E}_{H^1}(z_{0,r,h})$ or $\mathcal{E}_{N}(z_{0,r,h}) > 0.5 \mathcal{E}_{H^1}(z_{0,r,h})$, add $\mathcal M_z\cup\mathcal M_{z_0}$ to $\mathcal M$,
            \item if $\mathcal E_{\mathrm{DWR}}(\delta\overline{\varphi}_h) > 0.5 \mathcal{E}_J(z_{0,r,h},\overline{\varphi}_{0,h})$ or $E_r > 0.5 \mathcal{E}_J(z_{0,r,h},\overline{\varphi}_{0,h})$, add $\mathcal M_{\mathrm{DWR}}$ to $\mathcal M$;
        \end{itemize}
    \item if $\mathcal M$ is empty or the total number of degrees of freedom (DOFs) in $D^g_0$ exceeds 5 million, conclude. Otherwise, go to the next step;
    \item refine the elements in $\mathcal M$ to obtain the new (refined) mesh on $D^g_0$, and restrict the obtained mesh to the exact domain $D^g$ to obtain the new (refined) mesh on $D^g$;
    \item recompute the regularization parameter~$r$ from the updated mesh and go back to step~(a).
\end{enumerate}

\begin{remark}
    Since the goal-oriented defeaturing error estimator $\mathcal E_J$ is computed from $\mathcal E_{H^1}$ evaluated in the primal and in the dual solutions, the two first marking conditions in step~\ref{item:markingconditions} ensure that the numerical errors in the primal and dual solutions do not pollute the estimator $\mathcal E_J$. In addition, $\mathcal E_{\mathrm{DWR}}$ controls the second term in~\eqref{eq:decompositionerrorcontrol}, and the regularization error $E_r$ corresponds to the third term in~\eqref{eq:decompositionerrorcontrol}. The last marking criterion in step~\ref{item:markingconditions} thus ensures that the geometry simplification error dominates over both the numerical and regularization errors.
\end{remark}

For verification, we also compute the defeaturing errors $\left\|\overline{\E}_h-\overline{\E}_{0,h}\vert_{D^g}\right\|_{0,D^g}$, $\left\|z_h-z_{0,h}\vert_{H^1(D^g)}\right\|_{H^1(D^g)}$ and $|J_r(\overline{\varphi}_h)-J_r(\overline{\varphi}_{0,h}\vert_{D^g})|$ at every iteration, together with the inception voltage error $|\hat V_h-\hat V_{0,h}|$ and its non-regularized first-order approximation $|J(\overline{\varphi}_h)-J(\overline{\varphi}_{0,h}\vert_{D^g})|$.

\subsection{Numerical results}\label{ss:results}
In this section, we present the numerical results using the converged meshes and quantities obtained through the adaptive strategy described in Section~\ref{ss:adaptive}. In the first part, we assess the feature size dependence of the geometry defeaturing inception voltage error estimator $\mathcal{E}_J$ by fixing the shape of the feature and varying its size. In the second part, we assess the feature shape dependence of $\mathcal{E}_J$ by varying its aspect ratio. Indeed, the theory presented in Section~\ref{ss:estimatorfrom1storder} predicts that the effectivity index $\mathcal{I}_\mathrm{eff}$ should be constant with respect to the feature's size, but it does not provide any insight on the shape dependence of $\mathcal{I}_\mathrm{eff}$, which is investigated here numerically.

\subsubsection{Feature size dependence}\label{ss:sizenumtest}
In this numerical test, we consider $H=\displaystyle\frac{W}{2}$~mm for different values of $W=\displaystyle\frac{1}{2^k}$~mm, $k=0,1,2,3$. That is, we consider a family of features having the same shape but different sizes, parameterized by $k$. 

Figure~\ref{fig:1-size} reports the quantities directly involved in the goal-oriented inception voltage error estimation: panel~(a) compares the reference inception voltage defeaturing error $|\hat V_h-\hat V_{0,h}|$, its computable first-order approximation $|J_r(\delta\overline{\varphi}_h)|$, and the goal-oriented defeaturing estimator $\mathcal E_J(z_{0,r,h},\overline{\varphi}_{0,h})$; panel~(b) shows the corresponding effectivity index $\mathcal I_\mathrm{eff}$; panel~(c) decomposes $\mathcal E_J$ into the primal and dual $H^1$-defeaturing estimators $\mathcal E_{H^1}(\overline{\varphi}_{0,h})$ and $\mathcal E_{H^1}(z_{0,r,h})$ and the corrector contribution $|R(z_{0,r,h},\overline{\varphi}_{0,h})|$; and panel~(d) monitors the three terms entering the error decomposition~\eqref{eq:decompositionerrorcontrol} that must remain below~$|J_r(\delta\overline{\varphi}_h)|$, namely the DWR numerical-error estimator $\mathcal E_{\mathrm{DWR}}(\delta\overline{\varphi}_h)$ and the regularization error $E_r$.
Figure~\ref{fig:2-size} reports the auxiliary $H^1$-quantities feeding the goal-oriented estimator: the primal (left column) and dual (right column) $H^1$-defeaturing errors and their estimators (first row), the corresponding $H^1$-effectivity indices (second row), the residual $H^1$-numerical error estimators $\mathcal E_N$ on both the exact and the simplified geometries (third row), and the numerical convergence of the primal and dual $H^1$-defeaturing errors and estimators with respect to the number of degrees of freedom for $W=1$~mm (fourth row).

The results confirm both items~(i) and~(ii) stated above. First, $|J_r(\delta\overline{\varphi}_h)|$ approximates $|\hat V_h-\hat V_{0,h}|$ almost exactly across the different feature sizes, and both decay as $W^2$ when $W$ is halved. That is, the inception voltage error scales like the feature area, and the goal-oriented estimator $\mathcal E_J$ exhibits the same scaling and remains a tight upper bound. Second, the effectivity index $\mathcal I_\mathrm{eff}$ stays almost constant, between $1.03$ and $1.05$, as predicted by the theory presented in Section~\ref{ss:estimatorfrom1storder}. Panel~(c) further shows that $\mathcal E_J$ is largely dominated by the product of the primal and dual $H^1$-error estimators, while the corrector term $|R(z_{0,r,h},\overline{\varphi}_{0,h})|$ remains at machine-precision level (indeed, $V_D=0$ in this case). 

Finally, panel~(d) and the convergence study of Figure~\ref{fig:2-size} confirm that the adaptive strategy successfully drives the DWR numerical error $\mathcal E_{\mathrm{DWR}}$ and the regularization error $E_r$ well below $|J_r(\delta\overline{\varphi}_h)|$ for every $W$, so that the reported effectivity index reflects the geometry simplification error alone. The $H^1$-effectivity indices in Figure~\ref{fig:2-size} are themselves nearly constant ($\approx 1.34$), as expected from the literature, and they are remarkably the same for the primal and dual problems. Moreover, in the third row of Figure~\ref{fig:2-size}, we can see that the residual-based $H^1$-numerical error estimators $\mathcal E_N$ for the primal problems are driven to be about half an order of magnitude smaller than the corresponding $H^1$-defeaturing errors, as expected from the marking conditions in step~\ref{item:markingconditions} of the adaptive strategy. We observe however that this is not the case for the dual problem, for which $\mathcal E_N(z_{0,r,h})$ remains very large with respect to $\mathcal E_{H^1}(z_{0,r,h})$ across all values of $W$. This comes from the fact that the adaptive strategy has reached the forcing stopping criterion on the total number of degrees of freedom. The dual problem is indeed difficult to solve since $z_{0,r}$ exhibits a sharp layer around the critical part of the field line $\nu$. However, even though the numerical error in the dual solution is not fully under control, the per-DOF convergence panels in the last row confirm that the $H^1$-defeaturing errors and estimators have still reached their asymptotic plateau on the converged adaptive mesh for $W=1$~mm. Similar figures are obtained for the other values of $W$ (not reported here for conciseness), confirming that the effectivity indices are not polluted by the numerical errors.

\begin{figure}
\centering

\definecolor{magma1}{RGB}{45,17,96}
\definecolor{magma2}{RGB}{114,31,129}
\definecolor{magma3}{RGB}{182,54,121}
\definecolor{magma4}{RGB}{222,73,104}
\definecolor{magma5}{RGB}{251,133,28}
\definecolor{magma6}{RGB}{254,196,67}

\pgfplotstableread{
w inception_errors goal_defeaturing_errors goal_defeaturing_estimators goal_defeaturing_corrector_est goal_defeaturing_h1_est_primal goal_defeaturing_h1_est_dual dwr_est regularization_error_defeat eff_goal
1.2500000000e-01 1.7392457600e+00 1.7331024657e+00 1.7998202966e+00 5.7257233765e-16 1.1516620910e-02 1.5628024146e+02 6.0149201424e-04 1.3764108967e-02 1.0348280490e+00
2.5000000000e-01 6.9569063831e+00 6.9347799850e+00 7.2010878601e+00 2.9548398468e-15 2.3036203652e-02 3.1259872368e+02 4.6344563600e-03 5.5434161995e-02 1.0350991466e+00
5.0000000000e-01 2.7804410361e+01 2.7755683319e+01 2.8914294398e+01 4.3932153069e-15 4.6161830638e-02 6.2636801874e+02 3.2580146750e-02 2.2917284165e-01 1.0399175535e+00
1.0000000000e+00 1.1069748424e+02 1.1114443013e+02 1.1574371712e+02 7.9079422631e-15 9.2332585468e-02 1.2535522160e+03 9.3694713854e-02 9.4455784750e-01 1.0455857955e+00
}\datavswidth

\begin{tikzpicture}
\begin{groupplot}[
    group style={group size=2 by 2, horizontal sep=2.4cm, vertical sep=2.0cm},
    width=8cm, height=6cm,
    xmode=log,
    xlabel={Feature width $W$ [mm]},
    xticklabel style={/pgf/number format/fixed},
    grid=both, grid style={dashed,gray!30},
    legend cell align=left,
    legend style={font=\small, draw=none, fill=none},
    cycle list={
        {magma1},
        {magma2},
        {magma3},
        {magma4},
        {magma5},
        {magma6},
    },
]

\nextgroupplot[
    title={\textbf{(a)} Inception voltage defeaturing errors and estimators},
    ymode=log,
    ylabel={Error/Estimator [V]},
    legend pos=north west,
]
\addplot+[mark=square*,magma1] table[x=w, y=inception_errors]            {\datavswidth};
\addlegendentry{$\left|\hat V_h-\hat V_{0,h}\right|$}
\addplot+[mark=*,densely dotted,thick,magma4] table[x=w, y=goal_defeaturing_errors]     {\datavswidth};
\addlegendentry{$\left|J_r(\partial\overline{\varphi}_h)\right|$}
\addplot+[mark=diamond*,densely dashed,magma6,mark size=3] table[x=w, y=goal_defeaturing_estimators] {\datavswidth};
\addlegendentry{$\mathcal E_J(z_{0,r,h}, \overline{\varphi}_{0,h})$}

\nextgroupplot[
    title={\textbf{(b)} Effectivity index},
    ylabel={$\mathcal I_\mathrm{eff}$ [--]},
    ymode=normal,
    legend pos=north east,
    ymin=0.95, ymax=1.1,
]
\addplot+[mark=square*] table[x=w, y=eff_goal] {\datavswidth};
\addlegendentry{$\mathcal I_{\mathrm{eff}}$}

\nextgroupplot[
    title={\textbf{(c)} Defeaturing estimator components}, ymode=log, 
    ylabel={Components value [V]},
    legend pos=south east, legend style={yshift=20pt},
]
\addplot+[mark=square*,magma1] table[x=w, y=goal_defeaturing_estimators]    {\datavswidth};
\addlegendentry{$\mathcal E_J(z_{0,r,h}, \overline{\varphi}_{0,h})$}
\addplot+[mark=*,magma3] table[x=w, y=goal_defeaturing_h1_est_primal] {\datavswidth};
\addlegendentry{$\mathcal E_{H^1}(\overline{\varphi}_{0,h})$}
\addplot+[mark=triangle*,magma5,mark size=3] table[x=w, y=goal_defeaturing_h1_est_dual]   {\datavswidth};
\addlegendentry{$\mathcal E_{H^1}(z_{0,r,h})$}
\addplot+[mark=diamond*,magma6,mark size=3] table[x=w, y=goal_defeaturing_corrector_est] {\datavswidth};
\addlegendentry{$\left|R(z_{0,r,h}, \overline{\varphi}_{0,h})\right|$}

\nextgroupplot[
    title={\textbf{(d)} Total error components}, ymode=log,
    ylabel={Components value [V]},
    legend pos=south east,
]
\addplot+[mark=square*,magma1] table[x=w, y=goal_defeaturing_errors]     {\datavswidth};
\addlegendentry{$\left|J_r(\partial\overline{\varphi}_h)\right|$}
\addplot+[mark=*,magma4] table[x=w, y=dwr_est]                     {\datavswidth};
\addlegendentry{$\mathcal E_{\mathrm{DWR}}(\partial\overline{\varphi}_h)$}
\addplot+[mark=diamond*,magma6,mark size=3] table[x=w, y=regularization_error_defeat] {\datavswidth};
\addlegendentry{$E_r$}

\end{groupplot}
\end{tikzpicture}

\caption{Inception voltage defeaturing errors, estimators, and effectivity indices as functions of the feature width $W$ for a family of features of the same shape (but varying size).}\label{fig:1-size}
\end{figure}
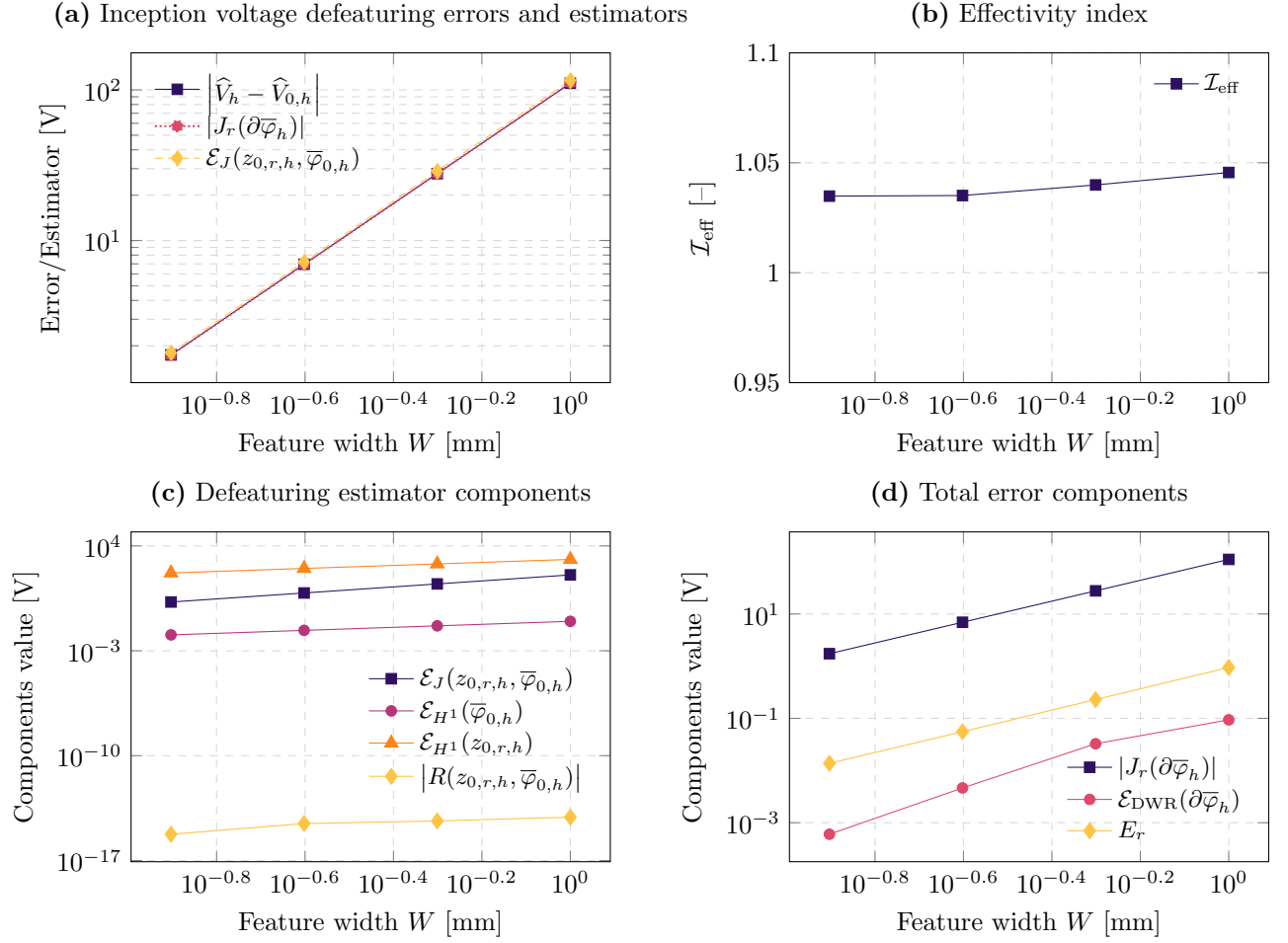
\input{Figure2}

\subsubsection{Feature shape dependence}
In the two following numerical tests, we consider a family of features having different shapes, or more precisely different $\displaystyle\frac{H}{W}$ ratios, parameterized by $k=0,1,2,3$. In the first test, we fix the feature's width to $W=0.125$~mm and vary its height $H=\displaystyle\frac{1}{2^{k+1}}$~mm; then in the second test, we fix the feature's height to $H=1$~mm and vary its width $W=\displaystyle\frac{1}{2^{k}}$~mm.

The results of the first test are given in Figures~\ref{fig:1-height} and~\ref{fig:2-height}, while the results of the second test are given in Figure~\ref{fig:1-width} and~\ref{fig:2-width}. The four figures follow the same layout as Figures~\ref{fig:1-size} and~\ref{fig:2-size}.
For both tests, we observe that the goal-oriented estimator $\mathcal E_J$ remains a reliable upper bound for $|\hat V_h-\hat V_{0,h}|$, it approximates well the variations of the defeaturing error across the different shapes, and the panels~(d) and the convergence rows of Figures~\ref{fig:2-height} and~\ref{fig:2-width} confirm that the adaptive strategy keeps the DWR and regularization errors well below the defeaturing error, so that the computed effectivity indices remain a property of the geometry simplification alone.
However, in contrast with Section~\ref{ss:sizenumtest}, we observe that the effectivity index has now some dependence on the aspect ratio $\displaystyle\frac{H}{W}$. 

In the test varying the feature's height $H$ (Figure~\ref{fig:1-height}), as $H$ grows from $0.0625$ to $0.5$~mm at fixed~$W=0.125$~mm, the effectivity index $\mathcal I_\mathrm{eff}$ of the goal-oriented defeaturing error estimator increases monotonically from about~$1.03$ (anisotropic feature, $\displaystyle\frac{H}{W}=0.5$) up to about~$1.38$ (very thin, tall feature, $\displaystyle\frac{H}{W}=4$): that is, the estimator $\mathcal E_J$ becomes more conservative for tall and narrow protrusions. The effectivity index however remains very low and seems to saturate as $H$ grows, so that the deterioration of the bound is mild even for very elongated features. Moreover, remarkably, the $H^1$-effectivity indices in Figure~\ref{fig:2-height} exhibit the opposite trend, slightly decreasing from~$\approx 1.34$ towards~$\approx 1.23$ as $H$ grows.

The mirror situation is observed in the test varying the feature's width $W$ (Figure~\ref{fig:1-width}): as $W$ grows from $0.125$ to $1$~mm at fixed~$H=0.5$~mm, the effectivity index $\mathcal I_\mathrm{eff}$ of the goal-oriented defeaturing errorestimator decreases monotonically from about~$1.38$ (very thin, tall feature, $\displaystyle\frac{H}{W}=4$) down to about~$1.05$ (anisotropic feature, $\displaystyle\frac{H}{W}=0.5$). Thus, both experiments consistently identify the feature's aspect ratio (rather than its size) as the geometric parameter that controls the sharpness of the bound, with $\mathcal I_\mathrm{eff}$ close to $1$ as the feature becomes anisotropic, but only with a mild deterioration for very elongated features. 

We also observe that the inception voltage defeaturing error itself is far more sensitive to~$H$ than to~$W$. Indeed, it grows by almost two orders of magnitude when the height is multiplied by $8$, but only by~$\sim 1.4$ when the width is increased by the same factor. This is in line with the physical intuition that the field enhancement produced by the protrusion is primarily governed by its height.
\begin{figure}
\centering

\definecolor{magma1}{RGB}{45,17,96}
\definecolor{magma2}{RGB}{114,31,129}
\definecolor{magma3}{RGB}{182,54,121}
\definecolor{magma4}{RGB}{222,73,104}
\definecolor{magma5}{RGB}{251,133,28}
\definecolor{magma6}{RGB}{254,196,67}

\pgfplotstableread{
w inception_errors goal_defeaturing_errors goal_defeaturing_estimators goal_defeaturing_corrector_est goal_defeaturing_h1_est_primal goal_defeaturing_h1_est_dual dwr_est regularization_error_defeat eff_goal
0.62500000000e-01 1.7392457343e+00 1.7331024677e+00 1.7998202965e+00 3.3652179116e-16 1.1516620910e-02 1.5628024145e+02 6.0149201569e-04 1.3764070753e-02 1.0348280643e+00
1.2500000000e-01 5.7251640761e+00 5.7063219873e+00 6.9022285043e+00 9.4709847529e-16 2.2531345198e-02 3.0633894442e+02 5.2349215693e-03 3.4706674565e-02 1.2055948812e+00
2.5000000000e-01 2.0547025101e+01 2.0500324434e+01 2.7111668293e+01 8.0645500435e-15 4.4653320608e-02 6.0715906284e+02 3.9800792875e-02 1.2793040866e-01 1.3194936084e+00
0.5000000000e+00 7.7409085011e+01 7.7538046071e+01 1.0710399365e+02 1.1641429981e-14 8.8700015202e-02 1.2074856290e+03 1.1805006311e-01 4.6654014496e-01 1.3836101232e+00
}\dataheightonlyvswidth

\begin{tikzpicture}
\begin{groupplot}[
    group style={group size=2 by 2, horizontal sep=2.4cm, vertical sep=2.0cm},
    width=8cm, height=6cm,
    xmode=log,
    xlabel={Feature height $H$ [mm]},
    xticklabel style={/pgf/number format/fixed},
    grid=both, grid style={dashed,gray!30},
    legend cell align=left,
    legend style={font=\small, draw=none, fill=none},
    cycle list={
        {magma1},
        {magma2},
        {magma3},
        {magma4},
        {magma5},
        {magma6},
    },
]

\nextgroupplot[
    title={\textbf{(a)} Inception voltage defeaturing errors and estimators},
    ymode=log,
    ylabel={Error/Estimator [V]},
    legend pos=north west,
]
\addplot+[mark=square*,magma1] table[x=w, y=inception_errors]            {\dataheightonlyvswidth};
\addlegendentry{$\left|\hat V_h-\hat V_{0,h}\right|$}
\addplot+[mark=*,densely dotted,thick,magma4] table[x=w, y=goal_defeaturing_errors]     {\dataheightonlyvswidth};
\addlegendentry{$\left|J_r(\partial\overline{\varphi}_h)\right|$}
\addplot+[mark=diamond*,densely dashed,magma6,mark size=3] table[x=w, y=goal_defeaturing_estimators] {\dataheightonlyvswidth};
\addlegendentry{$\mathcal E_J(z_{0,r,h}, \overline{\varphi}_{0,h})$}

\nextgroupplot[
    title={\textbf{(b)} Effectivity index},
    ylabel={$\mathcal I_\mathrm{eff}$ [--]},
    ymode=normal,
    legend pos=north west
]
\addplot+[mark=square*] table[x=w, y=eff_goal] {\dataheightonlyvswidth};
\addlegendentry{$\mathcal I_{\mathrm{eff}}$}

\nextgroupplot[
    title={\textbf{(c)} Defeaturing estimator components}, ymode=log, 
    ylabel={Components value [V]},
    legend pos=south east, legend style={yshift=20pt},
]
\addplot+[mark=square*,magma1] table[x=w, y=goal_defeaturing_estimators]    {\dataheightonlyvswidth};
\addlegendentry{$\mathcal E_J(z_{0,r,h}, \overline{\varphi}_{0,h})$}
\addplot+[mark=*,magma3] table[x=w, y=goal_defeaturing_h1_est_primal] {\dataheightonlyvswidth};
\addlegendentry{$\mathcal E_{H^1}(\overline{\varphi}_{0,h})$}
\addplot+[mark=triangle*,magma5,mark size=3] table[x=w, y=goal_defeaturing_h1_est_dual]   {\dataheightonlyvswidth};
\addlegendentry{$\mathcal E_{H^1}(z_{0,r,h})$}
\addplot+[mark=diamond*,magma6,mark size=3] table[x=w, y=goal_defeaturing_corrector_est] {\dataheightonlyvswidth};
\addlegendentry{$\left|R(z_{0,r,h}, \overline{\varphi}_{0,h})\right|$}

\nextgroupplot[
    title={\textbf{(d)} Total error components}, ymode=log,
    ylabel={Components value [V]},
    legend pos=south east,
]
\addplot+[mark=square*,magma1] table[x=w, y=goal_defeaturing_errors]     {\dataheightonlyvswidth};
\addlegendentry{$\left|J_r(\partial\overline{\varphi}_h)\right|$}
\addplot+[mark=*,magma4] table[x=w, y=dwr_est]                     {\dataheightonlyvswidth};
\addlegendentry{$\mathcal E_{\mathrm{DWR}}(\partial\overline{\varphi}_h)$}
\addplot+[mark=diamond*,magma6,mark size=3] table[x=w, y=regularization_error_defeat] {\dataheightonlyvswidth};
\addlegendentry{$E_r$}

\end{groupplot}
\end{tikzpicture}

\caption{Inception voltage defeaturing errors, estimators, and effectivity indices as functions of the feature height $H$ for a family of features of varying shape (fixed width, varying height).}\label{fig:1-height}
\end{figure}
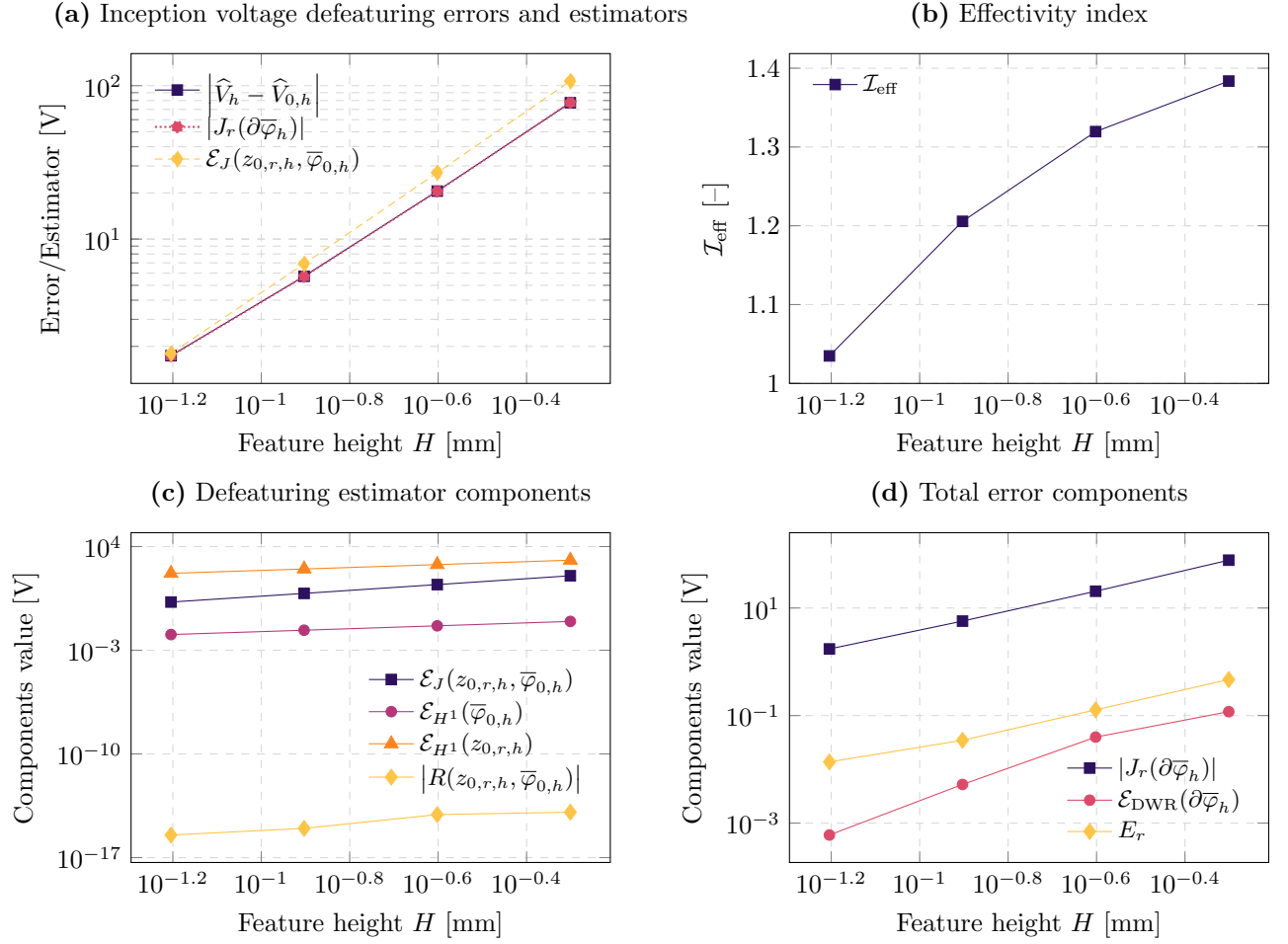
\input{figure2-height-only}
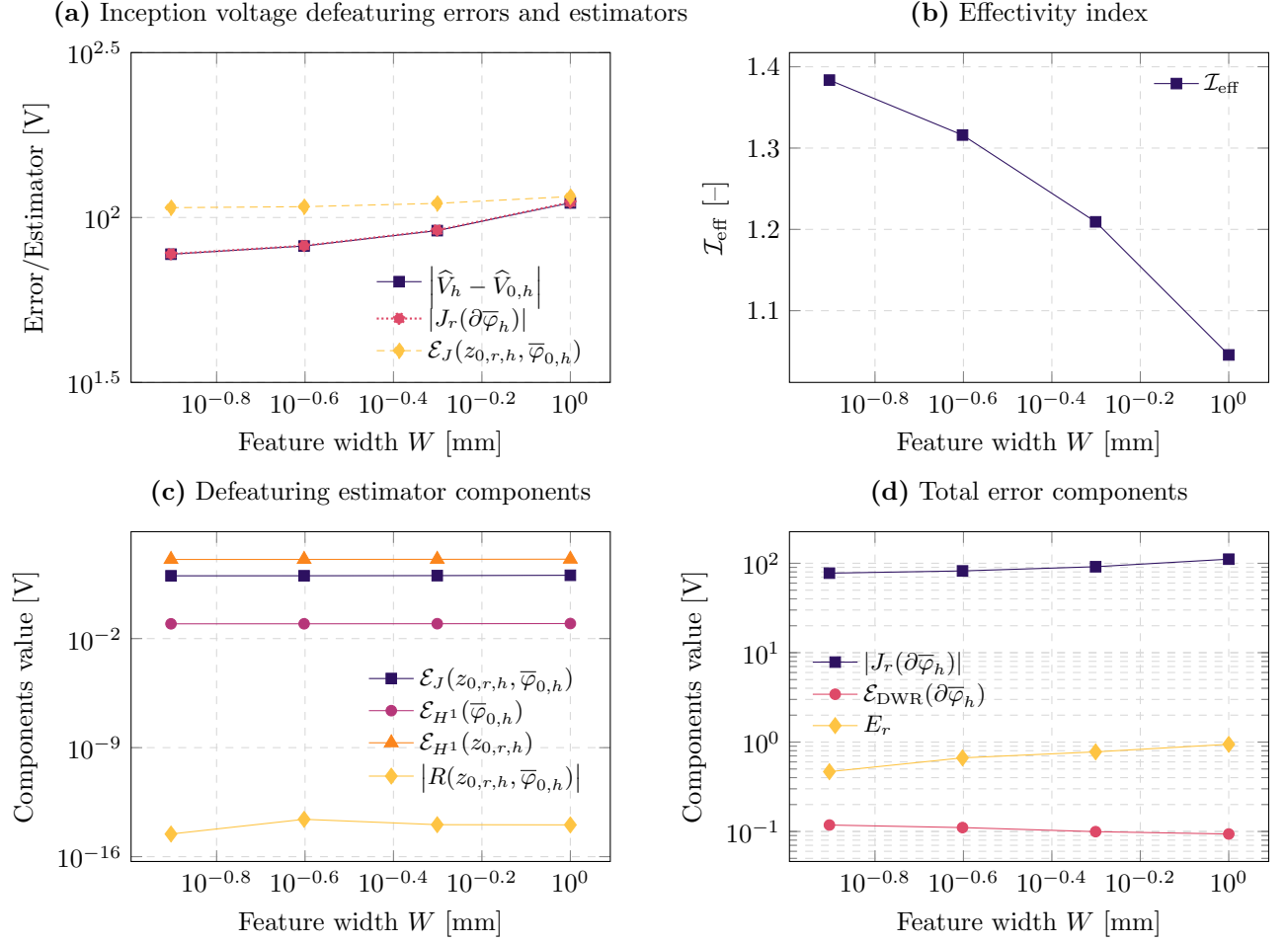
\begin{figure}
\centering

\definecolor{magma1}{RGB}{45,17,96}
\definecolor{magma2}{RGB}{114,31,129}
\definecolor{magma3}{RGB}{182,54,121}
\definecolor{magma4}{RGB}{222,73,104}
\definecolor{magma5}{RGB}{251,133,28}
\definecolor{magma6}{RGB}{254,196,67}

\pgfplotstableread{
w inception_errors goal_defeaturing_errors goal_defeaturing_estimators goal_defeaturing_corrector_est goal_defeaturing_h1_est_primal goal_defeaturing_h1_est_dual dwr_est regularization_error_defeat eff_goal
1.2500000000e-01 7.7409085217e+01 7.7538046427e+01 1.0710399377e+02 2.8529357808e-15 8.8700015220e-02 1.2074856301e+03 1.1805009125e-01 4.6654022549e-01 1.3836101211e+00
2.5000000000e-01 8.1942845582e+01 8.2106519297e+01 1.0783967986e+02 2.5114349884e-14 8.9098451136e-02 1.2103429239e+03 1.1051996454e-01 6.6594698193e-01 1.3160353304e+00
5.0000000000e-01 9.1245439971e+01 9.1487201072e+01 1.1033919744e+02 1.1322195039e-14 9.0143895625e-02 1.2240340478e+03 9.9527875558e-02 7.7686866838e-01 1.2092571144e+00
1.0000000000e+00 1.1069748421e+02 1.1114443066e+02 1.1574371715e+02 1.0979200757e-14 9.2332585468e-02 1.2535522163e+03 9.3694713674e-02 9.4455786264e-01 1.0455857960e+00
}\datawidthonlyvswidth

\begin{tikzpicture}
\begin{groupplot}[
    group style={group size=2 by 2, horizontal sep=2.4cm, vertical sep=2.0cm},
    width=8cm, height=6cm,
    xmode=log,
    xlabel={Feature width $W$ [mm]},
    xticklabel style={/pgf/number format/fixed},
    grid=both, grid style={dashed,gray!30},
    legend cell align=left,
    legend style={font=\small, draw=none, fill=none},
    cycle list={
        {magma1},
        {magma2},
        {magma3},
        {magma4},
        {magma5},
        {magma6},
    },
]

\nextgroupplot[
    title={\textbf{(a)} Inception voltage defeaturing errors and estimators},
    ymode=log,
    ylabel={Error/Estimator [V]},
    ymin=10^(1.5), ymax=10^(2.5),
    legend pos=south east
]
\addplot+[mark=square*,magma1] table[x=w, y=inception_errors]            {\datawidthonlyvswidth};
\addlegendentry{$\left|\hat V_h-\hat V_{0,h}\right|$}
\addplot+[mark=*,densely dotted,thick,magma4] table[x=w, y=goal_defeaturing_errors]     {\datawidthonlyvswidth};
\addlegendentry{$\left|J_r(\partial\overline{\varphi}_h)\right|$}
\addplot+[mark=diamond*,densely dashed,magma6,mark size=3] table[x=w, y=goal_defeaturing_estimators] {\datawidthonlyvswidth};
\addlegendentry{$\mathcal E_J(z_{0,r,h}, \overline{\varphi}_{0,h})$}

\nextgroupplot[
    title={\textbf{(b)} Effectivity index},
    ylabel={$\mathcal I_\mathrm{eff}$ [--]},
    ymode=normal,
    legend pos=north east
]
\addplot+[mark=square*] table[x=w, y=eff_goal] {\datawidthonlyvswidth};
\addlegendentry{$\mathcal I_{\mathrm{eff}}$}

\nextgroupplot[
    title={\textbf{(c)} Defeaturing estimator components}, ymode=log, 
    ylabel={Components value [V]},
    legend pos=south east, legend style={yshift=20pt},
]
\addplot+[mark=square*,magma1] table[x=w, y=goal_defeaturing_estimators]    {\datawidthonlyvswidth};
\addlegendentry{$\mathcal E_J(z_{0,r,h}, \overline{\varphi}_{0,h})$}
\addplot+[mark=*,magma3] table[x=w, y=goal_defeaturing_h1_est_primal] {\datawidthonlyvswidth};
\addlegendentry{$\mathcal E_{H^1}(\overline{\varphi}_{0,h})$}
\addplot+[mark=triangle*,magma5,mark size=3] table[x=w, y=goal_defeaturing_h1_est_dual]   {\datawidthonlyvswidth};
\addlegendentry{$\mathcal E_{H^1}(z_{0,r,h})$}
\addplot+[mark=diamond*,magma6,mark size=3] table[x=w, y=goal_defeaturing_corrector_est] {\datawidthonlyvswidth};
\addlegendentry{$\left|R(z_{0,r,h}, \overline{\varphi}_{0,h})\right|$}

\nextgroupplot[
    title={\textbf{(d)} Total error components}, ymode=log,
    ylabel={Components value [V]},
    legend pos=south west, legend style={yshift=40pt}
]
\addplot+[mark=square*,magma1] table[x=w, y=goal_defeaturing_errors]     {\datawidthonlyvswidth};
\addlegendentry{$\left|J_r(\partial\overline{\varphi}_h)\right|$}
\addplot+[mark=*,magma4] table[x=w, y=dwr_est]                     {\datawidthonlyvswidth};
\addlegendentry{$\mathcal E_{\mathrm{DWR}}(\partial\overline{\varphi}_h)$}
\addplot+[mark=diamond*,magma6,mark size=3] table[x=w, y=regularization_error_defeat] {\datawidthonlyvswidth};
\addlegendentry{$E_r$}

\end{groupplot}
\end{tikzpicture}

\caption{Inception voltage defeaturing errors, estimators, and effectivity indices as functions of the feature width $W$ for a family of features of varying shape (fixed height, varying width).}\label{fig:1-width}
\end{figure}
\input{figure2-width-only}

\section{Conclusion and outlook}\label{s:conclusion}
In this work, we presented a goal-oriented \textit{a posteriori} error estimation framework for assessing geometry simplification errors in electrostatic breakdown simulations, with a focus on inception voltage computations. The methodology combines a first-order approximation of the inception voltage error in terms of a linear functional~$J$ of the background electric field whose lack of regularity requires a rigorous treatment, with the recently proposed certified goal-oriented analysis-aware defeaturing estimator of~\cite{weder2025certified}. The theoretical framework, built upon the streamer integral model and goal-oriented error estimation theory, provides rigorous error bounds which are explicit in the feature's size. The proposed estimator requires computing quantities only in the simplified geometry and along the boundary of the removed feature, avoiding the need to mesh or solve the full detailed geometry.

The methodology was illustrated on a two-dimensional pin--plate benchmark with a small protrusion. To verify the reliability of the computed effectivity index, a single shared adaptive mesh and an adaptive coupling of the mollification width to the local mesh size were used to ensure that, on the converged mesh, the defeaturing error dominates over both the discretization and the regularization errors. Within this controlled setting, the numerical experiments demonstrated the following:
\begin{itemize}
    \item The first-order approximation $|J_r(\partial\overline{\varphi}_h)|$ tracks the reference inception voltage defeaturing error $|\hat V_h-\hat V_{0,h}|$ almost exactly across all considered feature sizes and shapes, confirming the validity of the first-order approximation;
    \item The goal-oriented defeaturing error estimator $\mathcal E_J$ remains a reliable upper bound with an effectivity index that stays low and of order unity (between about~$1.03$ and~$1.4$) across the entire family of considered features;
    \item When the feature shape is fixed and only its size is varied, the effectivity index remains constant, in agreement with the theory;
    \item When the feature's aspect ratio $\frac{H}{W}$ is varied, the effectivity index depends mildly but visibly on the shape, deteriorating for elongated features. This provides numerical insights into the otherwise unquantified shape-dependence of the constants in the error bound. 
    \item The inception voltage error is far more sensitive to the feature height than to its width, in line with the physical role of height in the local electric field enhancement.
\end{itemize}

This study is a first step towards the application of analysis-aware defeaturing to inception voltage predictions, and thus operates under quite restrictive assumptions. Several research directions therefore naturally arise: 
\begin{itemize}
    \item The assumption that the inception point coincides on the exact and on the simplified geometries is strong and should be relaxed, and a more general defeaturing error estimator should be derived without this assumption;
    \item Extensions to three-dimensional multi-electrode configurations with insulators and floating conductors are necessary for industrial applications; 
    \item The methodology should be validated on a broader range of geometric features and configurations to establish its robustness and practical applicability (e.g. scalability with the number of features and other application specific requirements). In particular, the regime of validity of the first-order approximation should be further investigated (for instance in configurations where the feature is not aligned with the opposite electrode, leading to distorted field lines), and the potential of higher-order approximations should be explored;
    \item A change of insulator geometry does not mathematically fall within the geometry simplification setting analyzed here since the computational domain in which Gauss' law~\eqref{eq:poissonmain} is solved is not modified. In this case, what changes is the definition of the permittivity~$\varepsilon$; we could therefore treat this situation as a parameter perturbation problem.
\end{itemize}

\appendix
\section{Proofs}\label{appendix}
\begin{lemma} \label{lemma:subharmonicgradient}
Let $\varphi:D\to\R$ be a $\mathcal C^3(D)$-harmonic function over some domain $D\subset\R^3$, i.e., $\Delta \varphi = 0$ in $D$. Then $|\nabla \varphi|$ is a sub-harmonic function over $D$, i.e., $\Delta(|\nabla \varphi|) \geq 0$ in $D$.
\end{lemma}
\begin{proof}
    Let us start by writing the expression of $|\nabla \varphi|^2$ explicitly:
    \begin{align}
        |\nabla \varphi|^2 = \sum_{i=1}^3 \left(\frac{\partial \varphi}{\partial x_i}\right)^2.
    \end{align}
    Then, for all $j\in\{1,2,3\}$,
    \begin{align}
        \frac{\partial}{\partial x_j}\left(|\nabla \varphi|^2\right) &= 2\sum_{i=1}^3 \frac{\partial \varphi}{\partial x_i}\frac{\partial^2 \varphi}{\partial x_j\partial x_i},\\
        \text{and thus } \frac{\partial^2}{\partial x_j^2}\left(|\nabla \varphi|^2\right) &= 2\sum_{i=1}^3\left[ \left(\frac{\partial^2 \varphi}{\partial x_j\partial x_i}\right)^2 + \frac{\partial \varphi}{\partial x_i}\frac{\partial}{\partial x_i}\left(\frac{\partial^2\varphi}{\partial^2 x_j}\right) \right].
    \end{align}
    Therefore, since $\varphi$ is harmonic,
    \begin{equation}
        \Delta(|\nabla \varphi|^2) = 2 \sum_{i=1}^3 \left[\sum_{j=1}^3 \left(\frac{\partial^2 \varphi}{\partial x_j\partial x_i}\right)^2 + \frac{\partial \varphi}{\partial x_i}\frac{\partial}{\partial x_i}\left(\Delta \varphi\right)\right] = 2 \sum_{i,j=1}^3 \left(\frac{\partial^2 \varphi}{\partial x_j\partial x_i}\right)^2 \geq 0.
    \end{equation}
    Consequently, $|\nabla \varphi|^2$ is sub-harmonic in $D$, and since the square root is a strictly increasing function in $[0,\infty)$, then $|\nabla \varphi|$ is also sub-harmonic in $D$.
\end{proof}

\vspace{1cm}
\begin{proof}[Proof of Proposition~\ref{prop:1storderapprox}]
In this proof, let us denote $\left\langle \displaystyle\frac{\mathrm dF}{\mathrm d\E}(\E), \nabla\delta\overline{\varphi}\right\rangle$ instead of $F'(\E)\cdot\nabla\delta\overline\varphi$ the G\^ateaux derivative of any functional $F\in \left(\left[L^2(D^g)\right]^3\right)^* \cong \left[L^2(D^g)\right]^3$ (up to isomorphism) in direction $\nabla\delta\overline\varphi$, evaluated at some electric field $\E\in\left[L^2(D^g)\right]^3$.

In equation~\eqref{eq:defomegahat}, $\hat\omega$ is defined by minimizing over the value of the streamer integral~$S$. This quantity depends on multiple arguments, all of which implicitly depend on the underlying background electric field. Let us therefore introduce the reduced streamer integral functional $\Sigma$ evaluated in the inception scaling $\hat\omega$: for all $\E\in\left[L^2(D^g)\right]^3$, let
\begin{equation}
    \Sigma(\E) := S\Big( \hat\omega(\E), \E, \nu, \hat T\big(\hat\omega(\E), \E\big)\Big).
\end{equation}
From~\eqref{eq:defomegahat}, we know that $\Sigma(\E) = K_c$ for all background electric fields $\E$. This means that for all~$\delta\E\in\left[L^2(D^g)\right]^3$, 
\begin{equation}
    \left\langle \frac{\mathrm d\Sigma}{\mathrm d\E}(\E), \delta\E\right\rangle = 0.
\end{equation}
This is thus in particular true for $\E = \overline{\E}_0$ and for $\delta\E=\nabla\delta\overline{\varphi}$,
\begin{equation}\label{eq:dSdFiszero}
    \left\langle \frac{\mathrm d\Sigma}{\mathrm d\E}(\overline{\E}_0), \nabla\delta\overline{\varphi}\right\rangle = 0.
\end{equation}

Using the assumption that $\hat T_0$ is differentiable at $\overline{\E}_0$ in the direction~$\nabla\delta\overline{\varphi}$, let us then further develop equation~\eqref{eq:dSdFiszero}. Recalling the streamer integral definition~\eqref{eq:defS}, and using Leibniz integral rule,
\begin{align}
    0 = &\,\left\langle \frac{\mathrm d\Sigma}{\mathrm d\E}(\overline{\E}_0), \nabla\delta\overline{\varphi}\right\rangle\\
    = &\, \left\langle \frac{\mathrm d}{\mathrm d\E}\left(\int_0^{\hat T_0} \alpha_\mathrm{eff}\Big(\hat\omega(\overline{\E}_0)\big|\overline{\E}_0\big(\x_0(t)\big)\big|\Big)\,\mathrm dt \right), \nabla\delta\overline{\varphi}\right\rangle\\
    = &\, \left\langle \frac{\mathrm d\hat T_0}{\mathrm d\E}, \nabla\delta\overline{\varphi}\right\rangle \alpha_\mathrm{eff}\Big(\hat\omega(\overline{\E}_0)\big|\overline{\E}_0\big(\x_0(\hat T_0)\big)\big|\Big) + \int_0^{\hat T_0} \Bigg\langle \frac{\mathrm d}{\mathrm d\E}\left(\alpha_\mathrm{eff}\Big(\hat\omega(\overline{\E}_0)\big|\overline{\E}_0\big(\x_0(t)\big)\big|\Big) \Bigg), \nabla\delta\overline{\varphi}\right\rangle\,\mathrm dt.\label{eq:sumequalszero}
\end{align}
From Assumption~\ref{as:allassumptions}.\ref{as:highlynonhomog}, $\alpha_\mathrm{eff}\Big(\hat\omega(\overline{\E}_0)\big|\overline{\E}_0\big(\x_0(\hat T_0)\big)\big|\Big)=0$, and thus the first term vanishes. Moreover, from Section~\ref{ss:streamerintegralmodel}, we know that $\alpha_\mathrm{eff}$ is differentiable in all points where it is not zero, i.e., for every value above the critical field~$E_c$. Since $\hat T_0\in(0,T^c]$, recalling the definition of $T^c$ from~\eqref{eq:Tmax}, then almost everywhere in the interval $(0,\hat T_0)$ (i.e., except on a subset of measure zero), $\alpha_\mathrm{eff}$ is differentiable and $\big|\E_0\big(\x_0(t)\big)\big|\neq 0$. Consequently, using the chain rule and the product rule,
\begin{align}
    0 = &\int_0^{\hat T_0} \Bigg\langle \frac{\mathrm d}{\mathrm d\E}\left(\alpha_\mathrm{eff}\Big(\hat\omega(\overline{\E}_0)\big|\overline{\E}_0\big(\x_0(t)\big)\big|\Big) \Bigg), \nabla\delta\overline{\varphi}\right\rangle\,\mathrm dt\\
    = &\int_0^{\hat T_0} \alpha_\mathrm{eff}'\Big(\hat\omega(\overline{\E}_0)\big|\overline{\E}_0\big(\x_0(t)\big)\big|\Big)\left[ \left\langle\frac{\mathrm d\hat\omega}{\mathrm d\E}(\overline{\E}_0), \nabla\delta\overline{\varphi}\right\rangle \big|\overline{\E}_0\big(\x_0(t)\big)\big| + \hat\omega(\overline{\E}_0) \left\langle\frac{\mathrm d}{\mathrm d\E}\left( \big|\overline{\E}_0\big(\x_0(t)\big)\big| \right), \nabla\delta\overline{\varphi}\right\rangle \right] \,\mathrm dt\\
    = &\left\langle\frac{\mathrm d\hat\omega}{\mathrm d\E}(\overline{\E}_0), \nabla\delta\overline{\varphi}\right\rangle \int_0^{\hat T_0} \alpha_\mathrm{eff}'\Big(\hat\omega(\overline{\E}_0)\big|\overline{\E}_0\big(\x_0(t)\big)\big|\Big) \big|\overline{\E}_0\big(\x_0(t)\big)\big|\,\mathrm dt \\
    &+ \hat\omega(\overline{\E}_0) \int_0^{\hat T_0} \alpha_\mathrm{eff}'\Big(\hat\omega(\overline{\E}_0)\big|\overline{\E}_0\big(\x_0(t)\big)\big|\Big) \frac{\overline{\E}_0\big(\x_0(t)\big) \cdot \nabla\delta\overline{\varphi}\big(\x_0(t)\big)}{\big|\overline{\E}_0\big(\x_0(t)\big)\big|} \,\mathrm d t. \label{eq:secondterm}
\end{align}
Note that the computation of the second term in the last step relies on Assumption~\ref{as:allassumptions},\ref{as:notdistortedfieldlines}, implying that $\x_0$ is independent from the considered electric fields $\E$ (i.e., the ones whose field line originating at $\x^\star$ is $\nu$).

Therefore, 
\begin{align}
    &-\left\langle\frac{\mathrm d\hat\omega}{\mathrm d\E}(\overline{\E}_0), \nabla\delta\overline{\varphi}\right\rangle \int_0^{\hat T_0} \alpha_\mathrm{eff}'\Big(\hat\omega(\overline{\E}_0)\big|\overline{\E}_0\big(\x_0(t)\big)\big|\Big) \big|\overline{\E}_0\big(\x_0(t)\big)\big|\,\mathrm dt \\
    = \,&\hat\omega(\overline{\E}_0) \int_0^{\hat T_0} \alpha_\mathrm{eff}'\Big(\hat\omega(\overline{\E}_0)\big|\overline{\E}_0\big(\x_0(t)\big)\big|\Big) \frac{\overline{\E}_0\big(\x_0(t)\big) \cdot \nabla\delta\overline{\varphi}\big(\x_0(t)\big)}{\big|\overline{\E}_0\big(\x_0(t)\big)\big|} \,\mathrm d t. \label{eq:lastone}
\end{align}
From Section~\ref{ss:streamerintegralmodel}, we recall that for all $\xi>E_c$, $\alpha_\mathrm{eff}'(\xi)>0$. Therefore, using again the definition of $\hat T_0$ and the fact that $\big|\E_0\big(\x_0(t)\big)\big|\neq 0$ almost everywhere in the interval $(0,\hat T_0)$, then we know that the integral on the left hand side of~\eqref{eq:lastone} is strictly positive. We can thus divide both sides by this quantity to obtain the result of Proposition~\ref{prop:1storderapprox}.

Finally, remark that the steps in this proof also demonstrate that all G\^ateaux derivatives involved exist, except for $\hat T_0$ whose differentiability is assumed.
\end{proof}

\vspace{1cm}
\begin{proof}[Proof of Proposition~\ref{prop:regularityJ}]
Note that from the assumptions on the effective ionization coefficient $\alpha_\mathrm{eff}$ from Section~\ref{ss:streamerintegralmodel}, and from the regularity of $\overline{\E}_0$, $\kappa(t)$ is bounded on the critical portion
\begin{equation}
    \hat\nu:=\displaystyle\left\{\x_0(t)\in D^g_0:0<t<\hat T_0\right\}=\displaystyle\left\{\x_0(t)\in D^g:0<t<\hat T_0\right\}
\end{equation}
of the field line. The last inequality is a consequence of Assumption~\ref{as:allassumptions}.\ref{as:hatomegaambiguity}.

Hence $J(\delta\varphi)$ is the integral, along the open curve $\hat\nu$, of $\nabla\delta\varphi$ tested against the smooth tangential weight~$\kappa(t)\mathbf{t}$. This evaluation makes sense provided the trace of $\nabla\delta\varphi$ on the curve~$\hat\nu$ exists in $L^2(\hat\nu)$. In $\mathbb{R}^3$, the restriction to a smooth curve (a manifold of codimension~$2$) is bounded from $H^{s}(D^g)$ to $L^2(\hat\nu)$ for any $s>1$, see e.g. ~\cite{adams2003sobolev,lions1972nonhomogeneous}. And in order to have $\nabla\delta\varphi\in H^{1+\epsilon}(D^g)$ for some $\epsilon>0$, we require $\delta\varphi\in H^{2+\epsilon}(D^g)$. Therefore, $J$ is well-defined as a linear functional on $H^{2+\epsilon}(D^g)$, i.e., $J\in H^{-2-\epsilon}(D^g)$.
\end{proof}

\section*{Acknowledgments}
The author would like to thank Dr. Christoph Winkelmann and Dr. Giacomo Garegnani for the many insightful discussions and valuable feedback throughout the development of this work.

\section*{Declaration of generative AI and AI-assisted technologies in the manuscript preparation process}
During the preparation of this work the author used Claude, a large language model trained by Anthropic, to help with code generation and to improve the readability and language of the manuscript. After using this tool, the author reviewed and edited the content as needed and takes full responsibility for the content of this publication.

\section*{Declaration of interest statement}
The author is an employee of ABB, a company that designs and manufactures electrical equipment relevant to the simulation methods presented in this work. The author declares no other competing financial interests or personal relationships that could have appeared to influence the work reported in this paper.

\addcontentsline{toc}{chapter}{Bibliography}
\bibliography{refs}
\bibliographystyle{ieeetr}

\end{document}